\documentclass{article}

\usepackage{graphicx}
\usepackage[left=3cm, top=2cm, right=3cm, bottom=3cm]{geometry}
\usepackage{listings}
\usepackage{float}
\usepackage{amssymb,amsmath,mathtools}
\usepackage{enumerate}
\usepackage[numbers,square,sort]{natbib}
\usepackage{wrapfig,caption}
\usepackage{framed}
\usepackage{tabstackengine}
\usepackage{braket}
\usepackage[nottoc]{tocbibind}\settocbibname{References}
\usepackage{mathrsfs}

\usepackage{indentfirst}
\usepackage{setspace}
\usepackage{bm}
\usepackage{authblk}

\usepackage{tikz}
\tikzset{every loop/.style={}}
\usetikzlibrary{arrows,shapes,decorations.markings,automata,backgrounds,petri,bending,calc}

\usepackage{tikz-cd} 
\tikzset{
    labl/.style={anchor=south, rotate=90, inner sep=.5mm}
}

\usepackage{enumitem}
\setlist[enumerate,1]{label=(\arabic*)}
\newlist{steplist}{enumerate}{1}
\setlist[steplist]{label={Step \arabic*:}, ref={Step \arabic*}}

\usepackage[utf8]{inputenc}
\usepackage[english]{babel} 

\usepackage{amsthm}
\newtheorem{thm}{Theorem}[section]
\newtheorem{lem}[thm]{Lemma}
\newtheorem{prop}[thm]{Proposition}
\newtheorem{cor}[thm]{Corollary}
\numberwithin{equation}{section}

\newtheorem*{resfin}{Theorem \ref{mainresult}}

\theoremstyle{definition}
\newtheorem{exmp}[thm]{Example} 
\newtheorem{rem}[thm]{Remark}

\newcommand{\N}{\mathbb{N}}
\newcommand{\Q}{\mathbb{Q}}

\newcommand{\Z}{\mathbb{Z}}

\newcommand{\Comm}{\operatorname{Comm}}
\newcommand{\GL}{\operatorname{GL}}
\newcommand{\acts}{\curvearrowright}
\newcommand{\tr}{\operatorname{tr}}

\makeatletter
\newcommand{\subjclass}[2][1991]{%
  \let\@oldtitle\@title%
  \gdef\@title{\@oldtitle\footnotetext{#1 \emph{Mathematics subject classification.} #2}}%
}
\newcommand{\keywords}[1]{%
  \let\@@oldtitle\@title%
  \gdef\@title{\@@oldtitle\footnotetext{\emph{Key words and phrases.} #1.}}%
}
\makeatother

\newcommand\freefootnote[1]{%
  \let\thefootnote\relax%
  \footnotetext{#1}%
  \let\thefootnote\svthefootnote%
}

\begin{document}
\title{Separability properties of higher-rank GBS groups}
\date{}
\author[$\dagger$]{Jone Lopez de Gamiz Zearra}
\author[$\star$]{Sam Shepherd}

\affil[$\dagger$]{Facultad de Economía y Empresa (Bilbao--Sarriko)\\
Universidad del Pa\'is Vasco EHU/UPV\\
Avenida Lehendakari Agirre 83, 48015 Bilbao, Spain}
\affil[$\star$]{Institut für Mathematische Logik und Grundlagenforschung\\ 
Fachbereich Mathematik und Informatik, Universität Münster\\
Einsteinstraße 62, 48149 Münster, Germany}

\maketitle
\begin{abstract}
A rank $n$ generalized Baumslag--Solitar group is a group that splits as a finite graph of groups such that all vertex and edge groups are isomorphic to $\mathbb{Z}^n$.
In this paper we classify these groups in terms of their separability properties.
Specifically, we determine when they are residually finite, subgroup separable and cyclic subgroup separable.



\end{abstract}
\freefootnote{MSC 2020:	20F65 (Primary) 20E26, 20E08 (Secondary)}
\section{Introduction and statement of results}\label{Introduction and statement of results}

If $G$ is a group, the \emph{profinite topology} $\mathcal{PT}(G)$ on $G$ is the topology on the underlying set of the group whose basic open sets are cosets of normal subgroups of finite index. This topology is Hausdorff if and only if the intersection of all finite-index normal subgroups is trivial, and in this case $G$ is said to be \emph{residually finite}. We say that a subgroup $H<  G$ is \emph{separable} in $G$ if $H$ is closed in $\mathcal{PT}(G)$, or equivalently, if $H$ is the intersection of finite-index subgroups (not necessarily normal) of $G$. The group $G$ is called \emph{subgroup separable} if every finitely generated subgroup of $G$ is separable, and \emph{cyclic subgroup separable} if every cyclic subgroup is separable. Evidently, a subgroup separable group is cyclic subgroup separable, and cyclic subgroup separability further implies residual finiteness. 

Residual properties are linked with the decidability of algorithmic problems. More concretely, finitely generated residually finite groups have decidable word problem, and finitely generated subgroup separable groups have decidable membership problem (also known as the generalized word problem).

Baumslag--Solitar groups $BS(m,n)= \langle a,t \mid t a^m t^{-1}=a^n\rangle$ were introduced by Baumslag--Solitar \cite{BaumslagSolitar} as simple examples of non-Hopfian groups and they have been extensively studied due to their exotic properties. In particular, the groups $BS(m,n)$ were completely characterized as far as residual finiteness is concerned in \cite{CollinsLevin} and \cite{Meskin}.

These groups are HNN extensions of $\mathbb{Z}$, so one may think of two natural ways of generalizing this class of groups. On the one hand, one may consider HNN extensions of the form $\langle K, t \mid t^{-1}At= B, \varphi \rangle$ where $K\cong \mathbb{Z}^n$, $A$ and $B$ have finite index in $K$ and $\varphi$ is the associated isomorphism $\varphi \colon A \to B$. Residual finiteness and subgroup separability for these groups have been studied in a series of papers (see, for instance, \cite{Shirvani} and \cite{AndreadakisRaptisVarsos86}). Cyclic subgroup separability, however, has been classified only for HNN extensions that are not ascending (see, for example, \cite{ChoonWoan}). It is also worth pointing out that I.J. Leary and A. Minasyan constructed in \cite{LearyMinasyan21} the first examples of CAT(0) groups that are not biautomatic by using groups of this form.

On the other hand, another way of generalizing Baumslag--Solitar groups is considering fundamental groups of finite graphs of groups with all vertex and edge groups isomorphic to $\mathbb{Z}$, known as \emph{generalized Baumslag--Solitar (GBS groups}). These groups have been deeply studied in relation to a number of properties, including JSJ decompositions, quasi-isometries, automorphisms and commensurability (see for example \cite{Whyte01,Forester03,Levitt07,CasalRuizKazachkovZakharov21}). There are also classifications of GBS groups according to residual finiteness \cite{Levitt15} and other residual properties \cite{Sokolov21}.

In this paper we merge both definitions and we define higher-rank generalized Baumslag--Solitar groups. More precisely, we define a \emph{rank $n$ generalized Baumslag--Solitar group} (which we abbreviate to \emph{rank $n$ GBS group}) to be a group $G$ that splits as a finite graph of groups such that all vertex and edge groups are isomorphic to $\mathbb{Z}^n$, and we study the separability properties in this wider class of groups.
We note that higher-rank GBS groups were studied in recent work of Button in relation to virtual hierarchical hyperbolicity \cite{Button22}.

In Section \ref{ResFin} we classify residually finite rank $n$ GBS groups. We use immersions of graphs of groups to find simpler subgroups of rank $n$ GBS groups and deduce conditions on the indices of the edge groups in the respective vertex groups. In this way, we show:

\begin{resfin}
Let $G$ be a rank $n$ GBS group. Then $G$ is residually finite if and only if either
\begin{itemize}
    \item $G \cong \mathbb{Z}^n \ast_{\varphi}$ is an ascending HNN extension (i.e. $\varphi \colon \mathbb{Z}^n \to \Z^n $ is a monomorphism); or
    \item $G$ is virtually $\mathbb{Z}^n$-by-free.
\end{itemize}
\end{resfin}

We then study subgroup separability in Section \ref{subsep}, and based on the work of Raptis--Varsos \cite{RaptisVarsos96} we conclude that $G$ is subgroup separable if and only if it is virtually $\mathbb{Z}^n$-by-free (see Theorem \ref{thm:subgpsep}).

We finally deal with the more difficult property of being cyclic subgroup separable in Section \ref{sec:cycsubsep}. Based on Theorem \ref{mainresult} and the fact that virtually $\mathbb{Z}^n$-by-free groups are subgroup separable, the remaining case to consider is when $G$ is an ascending HNN extension $\mathbb{Z}^n \ast_{\varphi}$.
Then $G$ is of the form
\[G=\langle A,t\mid tat^{-1}=\varphi(a), a\in A\rangle,\]
where $A\cong\Z^n$ and $\varphi\colon A\to A$ is a monomorphism.
Note that the image $\varphi(A)$ necessarily has finite index in $A$.

In this setting, it is not hard to check that $A$ is separable in $G$ if and only if $\varphi(A)=A$, and in this case $G$ is subgroup separable (Theorem \ref{thm:subgpsep}). In particular, in the case $A\cong\Z$ (when $G$ is a Baumslag--Solitar group), it follows that $G$ is cyclic subgroup separable if and only if $\varphi(A)=A$. When the rank of $A$ is greater than $1$, however, the study is more challenging and has been unresolved until now. 

We analyze the map $\varphi$ from the perspective of linear algebra. More concretely, given a basis of $A$, we write $\varphi$ as an integer matrix, and we study how properties of this matrix relate to properties of the group $G$.
In Lemma \ref{lem:eigenvalues'} we prove that if $\varphi$ has an integral eigenvalue $\lambda$ such that $|\lambda|>1$, then $G$ contains a $BS(1,\lambda)$ subgroup, so in particular $G$ is not cyclic subgroup separable.
However, the converse does not hold.

To obtain a complete characterization of when $G$ is cyclic subgroup separable, we work with certain $\varphi$-invariant subgroups of $A$ which correspond to the factorization $f_\varphi(x)=f_1(x)f_2(x)\dots f_l(x)$ of the characteristic polynomial $f_\varphi(x)$ into monic integer polynomials $f_i(x)$ that are irreducible over $\Q$. We characterize cyclic subgroup separability in terms of these polynomials $f_i(x)$:

\begin{thm}[Theorem \ref{thm:css}]
Let $G=\mathbb{Z}^n \ast_{\varphi}$ be an ascending HNN extension as described above.
The following are equivalent.
    \begin{enumerate}
        \item\label{item:introcycsubsep} $G$ is cyclic subgroup separable.
        \item\label{item:introequivx^ni} There is no prime $p$ and $1\leq i\leq l$ with $f_i(x)\equiv x^{n_i}$ {\rm{(mod $p$)}} (where $n_i$ is the degree of $f_i(x)$).
    \end{enumerate}    
\end{thm}

In the case when \ref{item:introequivx^ni} fails, we explicitly exhibit some cyclic subgroups that are not separable (Remark \ref{rem:exhibitsubgroup}).
We also apply the theorem to some examples at the end of Section \ref{sec:cycsubsep} (Examples \ref{firstexample} to \ref{lastexample}).

\tableofcontents

\section{Preliminaries}\label{Preliminaries}

\subsection{Bass--Serre theory and labeled graphs}\label{subsec:BassSerre}

We start with basic concepts of Bass-Serre theory and of GBS groups that have been used to study residual properties of GBS groups (see, for instance, \cite{Levitt07} and \cite{Levitt15}). We then extend these definitions to our general class of groups.

By a \emph{graph of groups} $(G(-),\Gamma)$ we mean a connected graph $(\Gamma, V\Gamma, E\Gamma, \iota, \tau)$ together with a function $G(-)$ which assigns to each $v\in V\Gamma$ a group $G_v = G(v)$ and to each edge $e\in E\Gamma$ a distinguished subgroup $G_e = G(e)$ of $G(\iota e)$ and an injective homomorphism $\rho_e \colon G_e \to G_{\tau e}$. We call the $G_v$, $v\in V\Gamma$, the \emph{vertex groups} and the $G_e$, $e \in E\Gamma$, the \emph{edge groups}. A \emph{finite graph of groups} is a graph of groups where the underlying graph is finite.

If $(G(-),\Gamma)$ is a graph of groups and $\Gamma_0 \subseteq \Gamma$ is a maximal tree, the associated \emph{fundamental group} $\pi(G(-), \Gamma, \Gamma_0)$ is the group presented as follows:
\begin{itemize}
\item[(1)] The generating set is $\{ t_e \mid e\in E\Gamma \} \cup \bigcup_{v \in V\Gamma} G_v$.
\item[(2)] The relations:
\begin{itemize}
\item the relations for $G_v$, for each $v\in V\Gamma$.
\item $t_e^{-1} g t_e = \rho_e(g)$ for all $e\in E\Gamma$, $g\in G_e \subseteq G_{\iota e}$.
\item $t_e = 1$ for all $e\in E\Gamma_0$.
\end{itemize}
\end{itemize}

It turns out that if $\Gamma_0$ and $\Gamma_1$ are two different maximal subtrees of $\Gamma$, then the fundamental groups $\pi(G(-),\Gamma,\Gamma_0)$ and $\pi(G(-),\Gamma,\Gamma_1)$ are isomorphic, so we do not need to specify the maximal subtree. The two simplest examples of fundamental groups of graphs of groups are amalgamated free products and HNN extensions.

The beauty of Bass-Serre Theory comes from the link that it provides between fundamental groups of graphs of groups and groups acting on simplicial trees. More precisely, if $G= \pi(G(-), \Gamma, \Gamma_0)$ is the fundamental group of a graph of groups $(G(-), \Gamma)$, then the group $G$ acts without inversion of edges on a tree $T$ such that the underlying graph of the quotient graph $T \slash G$ is isomorphic to $\Gamma$ and the stabilizers of the vertices and edges of the tree $T$ are conjugate to the groups $G_v$, $v\in V\Gamma$, and $G_e$, $e\in E\Gamma$, respectively. The tree $T$ is named the associated \emph{Bass-Serre tree}. Conversely, if $G$ acts on a simplicial tree without inversions, then $G$ is isomorphic to the fundamental group of a graph of groups.

\bigskip

A \emph{labeled graph} is a finite graph $\Lambda$  where each edge $e$ has two labels $\lambda_{\iota, e}, \lambda_{\tau, e} \in \mathbb{Z}\setminus \{0\}$. A vertex is \emph{terminal} if it has valence one. We say that $\lambda_{\iota, e}$ is the \emph{initial label of $e$}, $\lambda_{\tau, e}$ is the \emph{terminal label of $e$} and that $\lambda_{\iota, e}$ and $\lambda_{\tau, e}$ are \emph{the labels of $e$}. We visualize labeled graphs with the labels $\lambda_{\iota, e}$ and $\lambda_{\tau, e}$ pictured near the vertices $\iota e$ and $\tau e$, respectively.

Recall that a GBS group $G$ is the fundamental group of a finite graph of groups $\Gamma$ whose vertex and edge groups are all infinite cyclic. If we choose generators for edge and vertex groups, the inclusion maps and the monomorphisms $\rho_e \colon G_e \to G_{\tau e}$ are multiplications by non-zero integers, so an edge $e$ has labels $\lambda_{\iota, e}, \lambda_{\tau, e} \in \mathbb{Z}\setminus \{0\}$ describing the inclusion of the edge group $G_e$ into the vertex group $G_{\iota e}$ and the homomorphism $\rho_e$, respectively. Thus, we may visualize the graph of groups $\Gamma$ as a labeled graph.

Conversely, a connected labeled graph $\Lambda$ defines a graph of groups. The fundamental group $G$ may be represented as follows: choose a maximal subtree $\Lambda_0 \subseteq \Lambda$. There is one generator $a_v$ for each vertex $v\in V\Lambda$ and one generator $t_e$ for each $e$ that is not in $\Lambda_0$. Each edge $e$ of $\Lambda$ contributes to a relation. The relation is $(a_{\iota e})^{\lambda_{\iota, e}}= (a_{\tau e})^{\lambda_{\tau, e}}$ if $e$ is in $\Lambda_0$, and $t_{e}^{-1}(a_{\iota e})^{\lambda_{\iota, e}}t_{e}= (a_{\tau e})^{\lambda_{\tau, e}}$ if $e$ is not in $\Lambda_0$.

\bigskip

For a natural number $n$, a \emph{rank $n$ generalized Baumslag-Solitar group} (which we abbreviate to \emph{rank $n$ GBS group}) is a group $G$ that splits as a finite graph of groups $\Gamma$ such that all edge and vertex groups are isomorphic to $\Z^n$. In this case we can also associate a labeled graph to $G$, which we will call the \emph{associated labeled graph}, where $\lambda_{\iota, e}$ will be the index of the edge group $G_e$ in the vertex group $G_{\iota e}$ and $\lambda_{\tau, e}$ the index of $\rho_e(G_e)$ in $G_{\tau e}$.

For $n=1$ these two ways of associating a labeled graph to a GBS group may differ, as the latter one is obtained from taking the absolute value of the labels of the former one. Moreover, in this general setting we do not have a correspondence between labeled graphs and rank $n$ GBS groups, since there are multiple subgroups of a fixed index in $\mathbb{Z}^n$. Fortunately, this vague way of visualizing rank $n$ GBS groups is sufficient for the purposes of this article.

\bigskip

Coming back to Bass-Serre theory, the action of a group $G$ on a simplicial tree $T$ is \emph{minimal} if $T$ is the only $G$-invariant subtree of $T$. For rank $n$ GBS groups, this is equivalent to no label near a terminal vertex being equal to $1$ in the associated labeled graph. A graph of groups is said to be \emph{reduced} if, given an edge with distinct endpoints $v_1$ and $v_2$, the subgroup $G_e < G_{v_1}$ is proper and the inclusion $G_e \hookrightarrow G_{v_2}$ is proper. In particular, an HNN extension is always reduced. An amalgam $A \ast_C B$ is reduced if $C \neq A$ and $C \neq B$. Again, this can be easily stated in terms of the labeled graph as follows: a rank $n$ GBS group is reduced if any edge with $\lambda_{\iota, e}=1$ or $\lambda_{\tau, e}=1$ is a loop. In this case we also say that the labeled graph $\Gamma$ is \emph{reduced}. If a graph of groups is not reduced, then it can be made reduced by iteratively collapsing edges that make it non-reduced, known as \emph{elementary collapses} (see \cite{Forester02}). Finally, an HNN extension $A \ast_C$ is \emph{ascending} if $C=A$ (but the conjugate of $C$ by the stable letter might not equal $A$). Looking in the associated labeled graph, an ascending HNN extension will correspond to a labeled graph with just one edge, which is a loop and which has at least one label equal to $1$.

If $G$ is the fundamental group of a graph of groups, then any subgroup of $G$ will also act on the Bass–Serre tree, and this action will give a quotient graph of groups carrying that subgroup. In the case of a free action, then we know that the quotient graph is a cover of the original graph and that there is a correspondence between covers and subgroups.

Bass constructed in \cite{Bass93} the notions of morphisms, immersions and coverings for graphs of groups to recover the same correspondence for graphs of groups. The definitions are more technical than one might anticipate, so we will just present them in the case we are interested in here, i.e. when the vertex groups are free abelian.

If $(\bar{G}(-), \bar{\Gamma})$ and $(G(-),\Gamma)$ are two  finite graphs of groups with free abelian vertex groups, by a \emph{morphism} \[ \phi \colon (\bar{G}(-), \bar{\Gamma}) \to (G(-),\Gamma)\]
we understand\\[5pt]
(i) a graph morphism $\phi:\bar{\Gamma} \to \Gamma$;\\[3pt]
(ii) group homomorphisms
\[ \phi_{\bar{v}} \colon \bar{G}_{\bar{v}} \to G_{\phi{\bar{v}}}, \quad \phi_{\bar{e}} \colon \bar{G}_{\bar{e}} \to G_{\phi{\bar{e}}} \]
such that the obvious commutativity holds between the inclusions $G_{e} \hookrightarrow G_{\iota e}$, $ \rho_e \colon G_{e} \hookrightarrow G_{\tau e}$ and $\phi$;\\[3pt]
(iii) elements $\delta_{\iota,\bar{e}}\in G_{\iota\phi\bar{e}}$ and $\delta_{\tau,\bar{e}}\in G_{\tau\phi\bar{e}}$ for each $\bar{e}\in E\bar{\Gamma}$.

A morphism of graphs of groups induces a homomorphism of the fundamental groups \cite[Proposition 2.4]{Bass93}.

\begin{rem}
    In the notation of \cite[Definition 2.1]{Bass93}, we are taking the elements $\gamma_a$ to be trivial, while the elements $\gamma_e=\delta_e$ correspond to the elements $\delta_{\iota,\bar{e}}$ and $\delta_{\tau,\bar{e}}$ in (iii) above.
    In this way, our notion of morphism of graphs of groups is a special case of Bass' notion, but it will suffice for the arguments in this paper.
    Note also that equation (2.2) in \cite{Bass93} reduces to the commutativity assumption in (ii) above because we are assuming that the vertex groups are abelian.
\end{rem}

A morphism of graphs of groups as above is an \emph{immersion} if\\[5pt]
(i) the group homomorphisms $\phi_{\bar{v}}$ and $\phi_{\bar{e}}$ are injective;\\[3pt]
(ii) $\phi_{\bar{\iota} \bar{e}}(\bar{G}_{\bar{\iota} \bar{e}}) \cap G_{\phi \bar{e}} = \phi_{\bar{e}}(\bar{G}_{\bar{e}})$ and $\phi_{\bar{\tau}\bar{e}}(\bar{G}_{\bar{\tau}\bar{e}})\cap \rho_{\phi\bar{e}}(G_{\phi\bar{e}})=\rho_{\phi\bar{e}}(\phi_{\bar{e}}(\bar{G}_{\bar{e}}))$
for all edges $\bar{e}$ in $\bar{\Gamma}$;\\[3pt]
(iii) if $\phi\bar{v}=v$ and $\iota e=v$ then the map
\begin{equation}\label{immersioniota}
    \{\delta_{\iota,\bar{e}}\mid \bar{\iota}\bar{e}=\bar{v},\, \phi\bar{e}=e\}\to G_v/\phi_{\bar{v}}(\bar{G}_{\bar{v}})G_e
\end{equation}
is injective; similarly, if $\tau e=v$ then the map
\begin{equation}\label{immersiontau}
    \{\delta_{\tau,\bar{e}}\mid \bar{\tau}\bar{e}=\bar{v},\, \phi\bar{e}=e\}\to G_v/\phi_{\bar{v}}(\bar{G}_{\bar{v}})\rho_e(G_e)
\end{equation}
is injective.

The immersion is a \emph{covering} if the maps in (iii) are bijections. 
An immersion of graphs of groups induces an injective homomorphism $\bar{G}\hookrightarrow G$ of the fundamental groups \cite[Proposition 2.7]{Bass93}. If the immersion is actually a covering, the subgroup $\bar{G}$ is of finite index in $G$.

\begin{rem}
    Conditions (ii) and (iii) above correspond to condition (ii) in \cite[Definition 2.6]{Bass93}.
\end{rem}


If $G$ is a rank $n$ GBS group, when referring to it as a graph of groups $(G(-), \Gamma)$ we will assume that $\Gamma$ is the associated labeled graph. 
In Section \ref{ResFin}, the main way we will build immersions and coverings of graphs of groups (and thus obtain certain subgroups of rank $n$ GBS groups) will be by building immersions and coverings of labeled graphs and using the following lemma.
The one exception to this will be Lemma \ref{lem:loopandedge}, for which we will provide further explanation.

\begin{lem}\label{lem:induceimm}
    Let $G$ be a rank $n$ GBS group with labeled graph $\Gamma$, and let $\bar{\Gamma}$ be another labeled graph. Suppose $\phi:\bar{\Gamma}\to\Gamma$ is an immersion (resp. covering) of graphs that preserves the labels.
    Then $\phi$ can be extended to an immersion (resp. covering) of graphs of groups $\phi:(\bar{G}(-),\bar{\Gamma})\to(G(-),\Gamma)$. Moreover, $\phi$ corresponds to a (finite-index) subgroup $\bar{G}<G$, which is itself a rank $n$ GBS group with labeled graph $\bar{\Gamma}$.
\end{lem}
\begin{proof}
    Define the vertex and edge groups of $(\bar{G}(-),\bar{\Gamma})$ by $\bar{G}_{\bar{v}}=G_{\phi\bar{v}}$ and $\bar{G}_{\bar{e}}=G_{\phi\bar{e}}$ and let the maps $\phi_{\bar{v}}$ and $\phi_{\bar{e}}$ be identity maps. Similarly, let $\bar{G}_{\bar{e}}\hookrightarrow\bar{G}_{\bar{\iota}\bar{e}}$ and $\rho_{\bar{e}}:\bar{G}_{\bar{e}}\hookrightarrow\bar{G}_{\bar{\tau}\bar{e}}$ 
    be the inclusions induced by the corresponding inclusions in $(G(-),\Gamma)$.
    If $\iota e=v=\phi\bar{v}$, then there is at most one edge $\bar{e}$ with $\bar{\iota}\bar{e}=\bar{v}$ and $\phi\bar{e}=e$, so we may take $\delta_{\iota,\bar{e}}$ to be trivial and equation (\ref{immersioniota}) will be satisfied. The case $\tau e=v$ works similarly. So $\phi$ is an immersion of graphs of groups.
    If $\phi:\bar{\Gamma}\to\Gamma$ is a covering of graphs, then there must be exactly one edge $\bar{e}$ with $\bar{\iota}\bar{e}=\bar{v}$ and $\phi\bar{e}=e$, so in this case the maps (\ref{immersioniota}) will be bijections (and similarly the maps (\ref{immersiontau}) will be bijections), so $\phi$ will be a covering of graphs of groups.
\end{proof}

\subsection{Modular homomorphism for rank $n$ GBS groups}\label{Modular homomorphism}

Let us first recall the definition of the abstract commensurator of a group $G$. We consider the set $\Omega(G)$ of all isomorphisms between subgroups of finite index of $G$. Two such isomorphisms $\varphi_1 \colon H_1 \to K_1$ and $\varphi \colon H_2 \to K_2$ are equivalent, written $\varphi_1 \sim \varphi_2$, if there exists a subgroup $H$ of finite index in $G$ such that both $\varphi_1$ and $\varphi_2$ are defined on $H$ and ${\varphi_1}_{|H}= {\varphi_2}_{|H}$. The factor-set $\Omega(G) \slash \sim$ is then a group, called the \emph{abstract commensurator of $G$} and denoted by $\Comm(G)$.

Now let $G$ be a rank $n$ GBS group and let $T$ be the corresponding Bass-Serre tree. We will assume throughout that the action $G\acts T$ is minimal. Note that each edge stabilizer necessarily has finite index in its adjacent vertex stabilizers, hence all edge and vertex stabilizers are commensurable in $G$.
This also implies that $T$ is locally finite.

Fix a vertex $x\in VT$. The \emph{modular homomorphism} is the homomorphism $\Delta: G\to\Comm(G_x)$ to the abstract commensurator of the vertex stabilizer $G_x$ defined as follows.
For $g\in G$, the intersections $G_{g^{-1}x}\cap G_x$ and $G_x\cap G_{gx}$ are both finite-index subgroups of $G_x$, and we define $\Delta(g)$ to be (the equivalence class of) the isomorphism 
\begin{align}\label{Deltacon}
    G_{g^{-1}x}\cap G_x&\to G_x\cap G_{gx}\\\nonumber
    h&\mapsto ghg^{-1}.
\end{align}
We have $G_x\cong\Z^n$, so it is easy to see that $\Comm(G_x)$ can be naturally identified with $\GL(n,\Q)$.
We will thus consider $\Delta$ as a homomorphism $\Delta\colon G\to\GL(n,\Q)$.

This is a generalization of the modular homomorphism $G\to\Q^*$ that is defined for rank 1 GBS groups (e.g. in \cite{Levitt15}).
We note that the modular homomorphism for rank $n$ GBS groups was also defined in \cite{Button22}, where it was proved that $\Delta(G)$ is finite if and only if $G$ is virtually the direct product of $\Z^n$ with a free group (which in turn is equivalent to $G$ being virtually hierarchically hyperbolic).

The following proposition gives several equivalent conditions for a rank $n$ GBS group to be virtually $\Z^n$-by-free. We note that the equivalence with condition \ref{item:modcon} is an adaptation of the argument from \cite[Proposition 10.4]{LearyMinasyan21}.

\begin{prop}\label{prop:normalZn}
Let $G$ be a rank $n$ GBS group with Bass-Serre tree $T$. The following are equivalent:
\begin{enumerate}
    \item\label{item:normalZn} There exists a normal subgroup $\Z^n\cong N\triangleleft G$.
    \item\label{item:normalactt} There exists a normal subgroup $\Z^n\cong K\triangleleft G$ that acts trivially on $T$.
    \item\label{item:modcon} The image of the modular homomorphism $\Delta(G)<\GL(n,\Q)$ is conjugate to a subgroup of $\GL(n,\Z)$.
    \item\label{item:vZnfree} $G$ is virtually $\Z^n$-by-free.
\end{enumerate}
\end{prop}
\begin{proof}

    \ref{item:normalZn} $\Rightarrow$ \ref{item:normalactt}: First suppose that $N$ is contained in a vertex stabilizer $G_v$. Since $N$ is normal, it fixes every vertex in the orbit $G\cdot v$. Let $S\subset T$ be a finite subtree whose $G$-translates cover $T$.
    As $T$ is locally finite, there exists a finite-index normal subgroup $K\triangleleft N$ that fixes $S$.
    Passing to a further finite-index subgroup if necessary, we may assume that $K$ is characteristic in $N$, hence normal in $G$. It then follows that $K$ fixes every $G$-translate of $S$, hence it acts trivially on $T$. Moreover, $K\cong\Z^n$ because it has finite index in $N\cong\Z^n$.

    Now suppose that $N$ is not contained in any vertex stabilizer of $G$. Then there exists $h\in N$ acting hyperbolically on $T$, say it translates along an axis $\gamma\subset T$. As $N$ is abelian, it must stabilize $\gamma$, and since $N$ is normal in $G$ it must stabilize all $G$-translates of $\gamma$.
    If there exists $g\in G$ with $g\gamma\neq\gamma$, then $g\gamma\cap\gamma$ is stabilized by $N$ but not by $h$ (since $h$ translates along $\gamma$), a contradiction. Thus $G$ stabilizes $\gamma$. But then $T=\gamma$ by the minimality of $G\acts T$, and we deduce that the kernel $K$ of the $G$-action has index either 1 or 2 in each vertex stabilizer of $G$, hence $K\cong\Z^n$ is the required normal subgroup.

    \ref{item:normalactt} $\Rightarrow$ \ref{item:modcon}: The given subgroup $K$ acts trivially on $T$, so it must be contained in the vertex stabilizer $G_x$ (where $x$ comes from the definition of the modular homomorphism $\Delta$). Moreover, $K\cong\Z^n\cong G_x$, so there is $B\in\GL(n,\Q)$ with integer entries such that $K=BG_x$ (identifying $\Comm(G_x)$ with $\GL(n,\Q)$ as we did earlier).
    Given $g\in G$, the matrix $\Delta(g)\in\GL(n,\Q)$ represents the conjugation map (\ref{Deltacon}). As $K$ is normal in $G$ and is contained in every vertex stabilizer of $G$, we deduce that $\Delta(g)K=K$.
    Therefore $(B^{-1}\Delta(g)B)G_x=G_x$, and $B^{-1}\Delta(g)B\in\GL(n,\Z)$. This holds for all $g\in G$, so $B^{-1}\Delta(G)B<\GL(n,\Z)$ as required.

    \ref{item:modcon} $\Rightarrow$ \ref{item:normalZn}: Assume $C$ is an element in $\GL(n,\Q)$ with $C^{-1}\Delta(G)C<\GL(n,\Z)$.
    Let $\{g_1,\dots,g_t\}$ be a finite generating set for $G$.
    Let $B:=mC$ for an integer $m$ chosen so that $N:=BG_x$ is contained in $G_{g_i^{-1}x}\cap G_x$ for all $1\leq i\leq t$.
    Then, for each $1\leq i\leq t$, since $B^{-1}\Delta(g_i)B=C^{-1}\Delta(g_i)C\in\GL(n,\Z)$, we have
    $$\Delta(g_i)N=(BB^{-1}\Delta(g_i)B)G_x=BG_x=N.$$
    This implies that $g_i Ng_i^{-1}=N$, so $\Z^n\cong N\triangleleft G$ as required.

    \ref{item:normalactt} $\Rightarrow$ \ref{item:vZnfree}: The given subgroup $K$ acts trivially on $T$, so it is contained in every vertex stabilizer of $G$. Since all these stabilizers are isomorphic to $\Z^n$, and since $K\cong\Z^n$, we deduce that $K$ has finite index in each vertex stabilizer of $G$. This implies that the quotient $G/K$ acts on $T$ with finite vertex stabilizers, hence $G/K$ is virtually free. It follows that $G$ is virtually $\Z^n$-by-free.

    \ref{item:vZnfree} $\Rightarrow$ \ref{item:normalactt}: Let $G'<G$ be a finite-index subgroup that is $\Z^n$-by-free. Since $G'$ is itself a rank $n$ GBS group, we may apply \ref{item:normalZn} $\Rightarrow$ \ref{item:normalactt} to $G'$ to obtain a subgroup $\Z^n\cong K'\triangleleft G'$ acting trivially on $T$.
    If $K\triangleleft G$ is the kernel of the action $G\acts T$, we deduce that $K'<K<G_v$ for any vertex stabilizer $G_v$.
    As $K'$ and $G_v$ are both isomorphic to $\Z^n$, it follows that $K\cong\Z^n$.
\end{proof}

\section{Residual finiteness}\label{ResFin}

The goal of this section is to prove the following result.

\begin{thm}\label{mainresult}
Let $G$ be a rank $n$ GBS group. Then $G$ is residually finite if and only if either
\begin{itemize}
    \item $G \cong \mathbb{Z}^n \ast_{\varphi}$ is an ascending HNN extension (i.e. $\varphi \colon \mathbb{Z}^n \to \Z^n $ is a monomorphism); or
    \item $G$ is virtually $\mathbb{Z}^n$-by-free.
\end{itemize}
\end{thm}

Our arguments will involve working with subgroups of $G$ that are generated by pairs of adjacent vertex groups, so the following proposition will be useful.

\begin{prop}\label{prop1}
Let $G$ be a residually finite group with a subgroup $H<G$ of the form
\[ H= A_1 \ast_{\varphi} A_2,\]
where $A_1$ and $A_2$ are free abelian groups of rank $n$ and $\varphi \colon A_1^\prime \to A_2^\prime$ is an isomorphism with $A_i^\prime < A_i$ a proper subgroup of finite index for $i\in \{1,2\}$.

Then there is a finite-index normal subgroup $\bar{G}$ in $G$ such that $\bar{G} \cap A_1 < A_1^\prime$ and $\bar{G} \cap A_2 < A_2^\prime$.
\end{prop}

We need the following lemma before proving Proposition \ref{prop1}.

\begin{lem}\label{lemma1}
Under the above conditions, 
\[ \bigcap_{\bar{G} \triangleleft_{fi} G} (\bar{G} \cap A_1) A_1^\prime= A_1^\prime \quad \text{and} \quad \bigcap_{\bar{G} \triangleleft_{fi} G} (\bar{G} \cap A_2) A_2^\prime= A_2^\prime.\]
\end{lem}

\begin{proof}
By symmetry, it suffices to prove the first equation.
Clearly, the group $A_1^\prime$ is contained in $\bigcap_{\bar{G} \triangleleft_{fi} G} (\bar{G} \cap A_1) A_1^\prime$. Now suppose that there is an element $g\in\bigcap_{\bar{G} \triangleleft_{fi} G} (\bar{G} \cap A_1) A_1^\prime$ such that $g\notin A_1^\prime$. Pick an element $y \in A_2$ so that $y \notin A_2^\prime$ and define $x$ to be the element $[g,y]$.

Then, by the Normal Form Theorem for amalgamated free products, $x\neq 1$. Since $G$ is residually finite, there is a finite-index normal subgroup $G_0$ in $G$ such that $x \notin G_0$. 

Recall that $g \in \bigcap_{\bar{G} \triangleleft_{fi} G} (\bar{G} \cap A_1) A_1^\prime$, so in particular, $g\in (G_0 \cap A_1) A_1^\prime$ and we write
\[ g= g_0 z \quad \text{with} \quad g_0 \in (G_0 \cap A_1)\setminus \{1\}, \quad z \in A_1^\prime.\]
Since $A_1$ and $A_2$ are abelian, $z$ commutes with $g_0$ and $y$, thus \[x= [g_0 z, y]=[g_0,y].\]
The element $g_0$ lies in $G_0$ and $G_0$ is a normal subgroup of $G$. Hence, $x$ is an element of $G_0$ and this is a contradiction by our choice of $G_0$.
\end{proof}

\begin{proof}[Proof of Proposition \ref{prop1}]
By Lemma \ref{lemma1},
\[ \bigcap_{\bar{G} \triangleleft_{fi} G} (\bar{G} \cap A_1) A_1^\prime= A_1^\prime \quad \text{and} \quad \bigcap_{\bar{G} \triangleleft_{fi} G} (\bar{G} \cap A_2) A_2^\prime= A_2^\prime.\]
Thus, if $g\in A_1 \setminus A_1^\prime$, then there is $\bar{G}_0 \triangleleft_{fi} G $ such that $g\notin (\bar{G}_0 \cap A_1) A_1^\prime$. This means that $g\neq n a_1^\prime$ for all $n\in \bar{G}_0 \cap A_1$ and $a_1^\prime \in A_1^\prime$, i.e. $g {a_1^{\prime}}^{-1} \neq n$ for all $a_1^\prime \in A_1^\prime$, $n\in \bar{G}_0 \cap A_1$. Hence,
\[gA_1^\prime \cap (\bar{G}_0 \cap A_1) = \varnothing.\]
Similarly, if $h\in A_2 \setminus A_2^\prime$, there is $\tilde{G}_1 \triangleleft_{fi} G$ such that
\[hA_2^\prime \cap (\tilde{G}_1 \cap A_2) = \varnothing.\]

The group $A_1^\prime$ has finite index in $A_1$, so we may consider $\{ 1, g_1, \dots, g_{\alpha-1}\}$ a left transversal of $A_1^\prime$ in $A_1$. Similarly we consider a left transversal $\{1,h_1,\dots,h_{\beta -1}\}$ of $A_2^\prime$ in $A_2$.

Then, for $i\in \{1,\dots,\alpha-1\}$ there is $N_i \triangleleft_{fi} G$ such that \[g_i A_1^\prime \cap (N_i \cap A_1)= \varnothing,\]
and for $j\in \{1,\dots,\beta-1\}$ there is $K_j \triangleleft_{fi} G$ such that \[h_j A_2^\prime \cap (K_j \cap A_2)= \varnothing.\]
Let us define $\bar{G}$ to be
\[ \bar{G}= \bigcap_{i=1}^{\alpha-1} N_i \cap \bigcap_{j=1}^{\beta-1} K_j.\]
Then, $\bar{G}$ is normal and of finite index in $G$ because $N_i$ and $K_j$ are normal and of finite index in $\bar{G}$. It remains to check that $\bar{G} \cap A_1 < A_1^\prime$ and $\bar{G} \cap A_2 < A_2^\prime$. We show that $\bar{G} \cap A_1 < A_1^\prime$ since the other case is symmetric.

Suppose that there is $y\in \bar{G} \cap A_1$ such that $y \notin A_1^\prime$. Then $y\in g_i A_1^\prime$ for some $i\in \{1,\dots,\alpha-1\}$, so $y\in g_i A_1^\prime \cap (\bar{G}\cap A_1)$. In particular, 
\[ g_i A_1^\prime \cap (N_i \cap A_1) \neq \varnothing,\]
which is a contradiction.
\end{proof}

In the next two lemmas we use Proposition \ref{prop1} to show that residually finite rank $n$ GBS groups are virtually $\Z^n$-by-free, provided that all edge-loops in a reduced labeled graph have both labels equal to 1.

\begin{lem}\label{lemma2}
Suppose that $G$ is a residually finite rank $n$ GBS group with reduced labeled graph $\Gamma$. Then there is a covering of finite graphs of groups $(\bar{G}(-),\bar{\Gamma}) \to (G(-),\Gamma)$ such that, for each non-loop edge $e$ in $\Gamma$, all lifts of $e$ to $\bar{\Gamma}$ have both labels equal to $1$.
\end{lem}

\begin{proof}
 Let $e$  be a non-loop edge in $\Gamma$ with labels $x,y >1$ (the labels must be greater than 1 since $\Gamma$ is reduced). The two vertex groups incident to $e$, say $A_1$ and $A_2$, generate a subgroup $H$ as in Proposition \ref{prop1}. Let $\bar{G} \triangleleft_{fi} G$  be the subgroup provided by Proposition \ref{prop1}, $\bar{\Gamma}$ its associated labeled graph and $(\bar{G}(-),\bar{\Gamma}) \to (G(-),\Gamma)$ the corresponding covering of graphs of groups. Let $\bar{e}$ be a lift of $e$ to $\bar{\Gamma}$. Note that both labels of $\bar{e}$ correspond to the indices $[A_1 \cap \bar{G} \colon A_1^\prime \cap \bar{G}]$ and $[A_2 \cap \bar{G} \colon A_2^\prime \cap \bar{G}]$. But  by Proposition \ref{prop1} the groups $A_1^\prime \cap \bar{G}$ and $A_2^\prime \cap \bar{G}$ are precisely $A_1 \cap \bar{G}$ and $A_2 \cap \bar{G}$, so the indices are equal to $1$. The lemma follows by doing this for all non-loop edges in $\Gamma$ and intersecting the subgroups $\bar{G}$.
 
\end{proof}

			
				
				
		
	

\begin{lem}\label{lemma3}
Suppose that $G$ is a residually finite rank $n$ GBS group with reduced labeled graph $\Gamma$, such that all edge-loops in $\Gamma$ have both labels equal to $1$. Then $G$ is virtually $\mathbb{Z}^n$-by-free.
\end{lem}

\begin{proof}
By Lemma \ref{lemma2}, there is a finite-index subgroup $\bar{G}$ of $G$ and a corresponding covering of graphs of groups $(\bar{G}(-),\bar{\Gamma}) \to (G(-),\Gamma)$ such that all labels in $\bar{\Gamma}$ equal $1$. Let $\bar{T}$ be the Bass-Serre tree of $\bar{G}$. The kernel $K$ of the $\bar{G}$-action on $\bar{T}$ has to be contained in any vertex stabilizer $\bar{G}_v$. As all edge labels in $\bar{\Gamma}$ are $1$, it turns out that $\bar{G}_v$ also lies in $K$. As a consequence, $K= \bar{G}_v$, so $K$ is a free abelian group of rank $n$ that is normal in $\bar{G}$ and acts trivially on $\bar{T}$. Hence, by Proposition \ref{prop:normalZn}, $\bar{G}$ is virtually $\mathbb{Z}^n$-by-free, so $G$ is also virtually $\mathbb{Z}^n$-by-free.
\end{proof}

The case where an edge-loop has a label greater than 1 requires more careful analysis.
We start with a lemma concerning rank n GBS groups with circular labeled graph, and the ratios of the labels.

\begin{figure}[H]
\centering
\begin{tikzpicture}[scale=1.3]
    \foreach \x [count=\p] in {0,...,5} {
        \node[shape=circle,fill=black, scale=0.75] (\p) at (-\x*60:2) {};};
    \foreach \x [count=\p] in {1,...,4} {
        \draw (\x*60:2.4) node {$v_\p$};};
        \draw (5*60:2.4) node {$v_{m-1}$};
                \draw (0*60:3) node {$v_0= v_{m}$};
    \draw (1) arc (0:240:2);
    \draw[dashed] (2) arc (300:240:2);
    \draw (1) arc (0:240:2);
    \draw (2) arc (300:360:2);
    \foreach \x [count= \p] in {0,...,2} {
        \draw[gray] (15+ \x*60:2.4) node[below] {$r_{\p}$};};
        \foreach \x [count=\p] in {0,...,2} {
        \draw[gray] (45+ \x*60:2.4) node[below] {$q_{\p}$};};
\foreach \x [count= \p] in {0,...,2} {
        \draw[blue] (30+ \x*60:2.4) node[below] {$e_{\p}$};};
 \draw[gray] (-15:2) node[right] {$q_{m}$};
  \draw[blue] (-30:2) node[right] {$e_{m}$};
  \draw[gray] (-45:2.1) node[right] {$r_{m}$};
   \draw[gray] (-165:2) node[left] {$r_{4}$};
  \draw[gray] (-125:2) node[left] {$q_{4}$};
    \draw[blue] (-145:2) node[left] {$e_{4}$};
\end{tikzpicture}
\caption{}\label{fig:loop}
\end{figure}

\begin{lem}\label{lemmaloop}
Let $G$ be a residually finite rank $n$ GBS group with labeled graph $\Gamma$ given by Figure \ref{fig:loop}. Assume that $r_i \neq 1$ for some $i\in \{1,\dots, m\}$ and $q_j \neq 1$ for some $j\in \{1,\dots,m\}$. 
Then
\[ \frac{q_1}{r_1}\frac{q_2}{r_2}\cdots\frac{q_{m-1}}{r_{m-1}}\frac{q_m}{r_m}=1.\]
\end{lem}

\begin{proof}
We can assume that the graph of groups is reduced since the number $\frac{q}{r}=\frac{q_1}{r_1}\frac{q_2}{r_2}\cdots\frac{q_{m-1}}{r_{m-1}}\frac{q_m}{r_m}$ is invariant under taking elementary collapses of edges. Since $r_i \neq 1$ and $q_i \neq 1$ for some $i$, we deduce that $\Gamma$ does not consist of a single edge with labels $1$ and $q\neq 1$. 
If $\Gamma$ consists of a single edge with labels $q\neq1$ and $r\neq1$, then we can take a degree 2 covering of labeled graphs $\bar{\Gamma}\to\Gamma$, and $\bar{\Gamma}$ will correspond to a finite-index subgroup of $G$ by Lemma \ref{lem:induceimm} (which is also residually finite). Then $\bar{\Gamma}$ would be a reduced labeled graph which is also a circle, and which has the ratio $\frac{q^2}{r^2}$, but since $\frac{q}{r}$ is $1$ if and only if $\frac{q^2}{r^2}$ is $1$, we can reduce to the case where $\Gamma$ is a reduced labeled graph with at least two edges (so no edge-loops).
Hence, by Lemma \ref{lemma3}, $G$ is virtually $\mathbb{Z}^n$-by-free.

Let $T$ be the Bass--Serre tree for $G$.
Let $\tilde{v}_0,\tilde{e}_1,\tilde{v}_1,\tilde{e}_2,\dots, \tilde{e}_{m},\tilde{v}_m$ be a lift to $T$ of the path $v_0,e_1,v_1,e_2,\dots,e_m,v_m$ in $\Gamma$, and let $\Z^n\cong K\triangleleft G$ be the normal subgroup given by Proposition \ref{prop:normalZn} that acts trivially on $T$. 
Then, $K$ is a subgroup of $G_{\tilde{e}_i}$ for all $i\in \{1,\dots, m\}$ and
\[ |G_{\tilde{v}_j} \colon K|= r_{j+1} |G_{\tilde{e}_{j+1}}\colon K| \quad \text{for} \quad j\in \{0,\dots,m-1\} \quad \text{and}\]
\[ |G_{\tilde{v}_l} \colon K|= q_{l} |G_{\tilde{e}_{l}}\colon K| \quad \text{for} \quad l\in \{1,\dots,m\}.\]
Hence, \[|G_{\tilde{v}_j} \colon K |= \frac{q_j}{r_j}|G_{\tilde{v}_{j-1}} \colon K |\] for each $j\in \{1,\dots,m\}$. Observe that $\tilde{v}_m= g \cdot \tilde{v}_0$ for some $g\in G$, so $G_{\tilde{v}_m}= G_{\tilde{v}_0}^{g}$. 
Since $K$ is normal in $G$, we get that $|G_{\tilde{v}_m} \colon K|=|G_{\tilde{v}_0} \colon K|$. In conclusion,
\[ |G_{\tilde{v}_0} \colon K |= |G_{\tilde{v}_m} \colon K | =\frac{q_m}{r_m}\frac{q_{m-1}}{r_{m-1}}\cdots\frac{q_1}{r_1}|G_{\tilde{v}_0} \colon K |,\]
and the lemma follows.
\end{proof}
\begin{figure}[H]
	\centering
		\begin{tikzpicture}[auto,node distance=2cm,
			thick,every node/.style={},
			every loop/.style={looseness=40},
			]
			\tikzstyle{label}=[draw=none,font=\large]
			\tikzstyle{node}=[circle,fill,draw,font=\small]
			
			\begin{scope}
				\node[node] (1) {};				
				
				\path
				(1) edge [in=225,out=135,loop] (1)
				(1) edge [in=45,out=-45,loop](1);
				
				\node[label] at (-1.5,1.3) {$q$};
				\node[label] at (-1.5,-1.3) {$1$};
				\node[label] at (1.5,1.3) {$x$};
				\node[label] at (1.5,-1.3) {$y$};
			\end{scope}
		
		\begin{scope}[shift={(6,0)}]
			\node[node] (1) {};	
			\node[node] (2) at (2,0){};			
			
			\path
			(1) edge [in=225,out=135,loop] (1)
			(1) edge (2);
			
			\node[label] at (-1.5,1.3) {$q$};
			\node[label] at (-1.5,-1.3) {$1$};
			\node[label] at (.5,.3) {$r$};
			\node[label] at (1.5,.3) {$s$};	
	\path		
(2) edge [in=45,out=-45,loop](2);

\node[label] at (3.5,1.3) {$x$};
\node[label] at (3.5,-1.3) {$y$};
		\end{scope}
			
		\end{tikzpicture}
	
	\caption{}\label{fig:doubleloop}
\end{figure}

\begin{lem}\label{lemmadoubleloop}
Suppose that $G$ is a rank $n$ GBS group with reduced labeled graph $\Gamma$ given by one of the two graphs in Figure \ref{fig:doubleloop}, with $q\neq 1$. Then $G$ is not residually finite.    
\end{lem}

\begin{proof}
Suppose that $G$ is residually finite.
Let $\bar{\Gamma}$ be one of the labeled graphs shown in Figure \ref{fig:rectangle}. There is an obvious immersion $\phi:\bar{\Gamma} \to \Gamma$, respecting the labels, which induces an immersion of graphs of groups by Lemma \ref{lem:induceimm}. 
As a result, there is a subgroup $\bar{G}<G$, which is a rank $n$ GBS group with associated labeled graph $\bar{\Gamma}$. We are assuming that $G$ is residually finite, so $\bar{G}$ is also residually finite.
Then, from Lemma \ref{lemmaloop} we deduce that $q=1$, a contradiction.

\end{proof}

\begin{figure}[H]
	\centering
		\begin{tikzpicture}[auto,node distance=2cm,
			thick,every node/.style={},
			every loop/.style={looseness=40},
			]
			\tikzstyle{label}=[draw=none,font=\large]
			\tikzstyle{node}=[circle,fill,draw,font=\small]
			
			\node[node] (1) {};
			\node[node] at (2,0){};
			\node[node] at (4,0){};
			\node[node] at (4,-2){};
			\node[node] at (0,-2){};

			\draw
			(0,0)--(4,0)--(4,-2)--(0,-2)--(0,0);
			
			\node[label] at (.5,.3) {$1$};
			\node[label] at (1.5,.3) {$q$};
			\node[label] at (2.5,.3) {$1$};
			\node[label] at (3.5,.3) {$q$};
			\node[label] at (.5,-2.3) {$1$};
			\node[label] at (3.5,-2.3) {$q$};
			\node[label] at (-.3,-1.5) {$x$};
			\node[label] at (-.3,-.5) {$y$};
			\node[label] at (4.3,-1.5) {$x$};
			\node[label] at (4.3,-.5) {$y$};	
			
			\begin{scope}[shift={(7,0)}]
\node[node] (1) {};
\node[node] at (2,0){};
\node[node] at (4,0){};
\node[node] at (4,-2){};
\node[node] at (0,-2){};
\node[node] at (4,-4){};
\node[node] at (0,-4){};
\node[node] at (4,-6){};
\node[node] at (0,-6){};

\draw
(0,0)--(4,0)--(4,-6)--(0,-6)--(0,0);

\node[label] at (.5,.3) {$1$};
\node[label] at (1.5,.3) {$q$};
\node[label] at (2.5,.3) {$1$};
\node[label] at (3.5,.3) {$q$};
\node[label] at (.5,-6.3) {$1$};
\node[label] at (3.5,-6.3) {$q$};
\node[label] at (-.3,-3.5) {$x$};
\node[label] at (-.3,-2.5) {$y$};
\node[label] at (4.3,-3.5) {$x$};
\node[label] at (4.3,-2.5) {$y$};
\node[label] at (-.3,-.5) {$r$};
\node[label] at (-.3,-1.5) {$s$};
\node[label] at (-.3,-4.5) {$s$};
\node[label] at (-.3,-5.5) {$r$};
\node[label] at (4.3,-.5) {$r$};
\node[label] at (4.3,-1.5) {$s$};
\node[label] at (4.3,-4.5) {$s$};
\node[label] at (4.3,-5.5) {$r$};
			\end{scope}	
		\end{tikzpicture}
	\caption{}\label{fig:rectangle}
\end{figure}

\begin{figure}[H]
	\centering
	\begin{tikzpicture}[auto,node distance=2cm,
		thick,every node/.style={},
		every loop/.style={looseness=40},
		]
		\tikzstyle{label}=[draw=none,font=\large]
		\tikzstyle{node}=[circle,fill,draw,font=\small]
		
		\node[node] (1) {};	
		\node[node] (2) at (2,0){};			
		
		\path
		(1) edge [in=225,out=135,loop] (1)
		(1) edge (2);
		
		\node[label] at (-1.5,1.3) {$q$};
		\node[label] at (-1.5,-1.3) {$1$};
        \node[label, gray] at (-2.5,0) {$e_1$};
        \node[label, gray] at (1,-0.3) {$e_2$};
        \node[label, gray] at (0,-0.4) {$v_1$};
        \node[label, gray] at (2,-0.4) {$v_2$};
		\node[label] at (.5,.3) {$r$};
		\node[label] at (1.5,.3) {$s$};			
	\end{tikzpicture}
	
	\caption{}\label{fig:loopandedge}
\end{figure}

\begin{lem}\label{lem:loopandedge}
	Suppose that $G$ is a rank $n$ GBS group with reduced labeled graph $\Gamma$ as in Figure \ref{fig:loopandedge}, with $q\neq 1$.
	Then $G$ is not residually finite.
\end{lem}
\begin{proof}
Since $\Gamma$ is reduced, we have $s \neq 1$.
Let $\bar{\Gamma}$ be the labeled graph shown in Figure \ref{fig:doubleloopbar}. Then there is an obvious map $\phi \colon \bar{\Gamma} \to \Gamma$, with $\phi(w_1)=\phi(w_3)=v_1$ and $\phi(w_2)=v_2$. This map respects the labels at $w_1$ and $w_3$ but not at $w_2$. 

We extend this to an immersion of graphs of groups $\phi:(\bar{G}(-),\bar{\Gamma}) \to (G(-),\Gamma)$ as follows.
The construction is similar to Lemma \ref{lem:induceimm}, but with a modification at the vertex group $\bar{G}_{w_2}$.
Define $\bar{G}_{w_2}=G_{e_2}<G_{v_2}$ and let $\phi_{w_2}$ be the inclusion $\bar{G}_{w_2}\hookrightarrow G_{v_2}$. Define all other vertex and edge groups in $(\bar{G}(-),\bar{\Gamma})$ to equal the corresponding vertex and edge groups in $(G(-),\Gamma)$ (and let the maps $\phi_{\bar{v}}, \phi_{\bar{e}}$ be identity maps).
Similarly, let all the maps $\bar{G}_{\bar{e}}\hookrightarrow\bar{G}_{\bar{\iota}\bar{e}}$ and $\rho_{\bar{e}}:\bar{G}_{\bar{e}}\hookrightarrow\bar{G}_{\bar{\tau}\bar{e}}$ 
be the inclusions induced by the corresponding inclusions in $(G(-),\Gamma)$.    
Suppose $e_2$ is oriented so that $\iota e_2=v_2$.
Take $\delta_{\iota,e''_2}$ to be some element of $G_{v_2}\setminus G_{e_2}$. Let all other elements $\delta_{\iota,\bar{e}}$ and $\delta_{\tau,\bar{e}}$ be trivial.

It is clear that the above data defines a morphism $\phi:(\bar{G}(-),\bar{\Gamma}) \to (G(-),\Gamma)$, and it is also clear that it satisfies properties (i) and (ii) of being an immersion (see Subsection \ref{subsec:BassSerre}). It remains to check condition (iii) of being an immersion.
By construction, the map $\{\delta_{\iota,e'_2},\delta_{\iota,e''_2}\}\to G_{v_2}/\phi_{w_2}(\bar{G}_{w_2})G_{e_2}=G_{v_2}/G_{e_2}$ is injective (remember $s\neq1$). All other instances of the maps from (\ref{immersioniota}) and (\ref{immersiontau}) are injective because their domains are singletons.
Thus $\phi$ satisfies condition (iii) of being an immersion.

As a result, we obtain a subgroup $\bar{G}<G$ with associated labeled graph $\bar{\Gamma}$. After applying an elementary collapse to $\bar{\Gamma}$, we obtain the right-hand graph from Figure \ref{fig:doubleloop}, so Lemma \ref{lemmadoubleloop} implies that $\bar{G}$ is not residually finite. Hence $G$ is not residually finite.

\end{proof}
\begin{figure}[H]
	\centering
	\begin{tikzpicture}[auto,node distance=2cm,
		thick,every node/.style={},
		every loop/.style={looseness=40},
		]
		\tikzstyle{label}=[draw=none,font=\large]
		\tikzstyle{node}=[circle,fill,draw,font=\small]
		
			\node[node] (1) {};	
			\node[node] (2) at (2,0){};	
			\node[node] (3) at (4,0){};		
			
			\path
			(1) edge [in=225,out=135,loop] (1)
			(1) edge (2)
			(2) edge (3);
			
			\node[label] at (-1.5,1.3) {$q$};
			\node[label] at (-1.5,-1.3) {$1$};
			\node[label] at (.5,.3) {$r$};
			\node[label] at (1.5,.3) {$1$};	
			\node[label] at (2.5,.3) {$1$};
			\node[label] at (3.5,.3) {$r$};
   \node[label, gray] at (0,-0.4) {$w_1$};
   \node[label, gray] at (2,-0.4) {$w_2$};
   \node[label, gray] at (4,-0.4) {$w_3$};
			\path		
			(3) edge [in=45,out=-45,loop](3);
			
			\node[label] at (5.5,1.3) {$q$};
			\node[label] at (5.5,-1.3) {$1$};

            \node[label, gray] at (-2.5,0) {$e'_1$};
            \node[label, gray] at (6.5,0) {$e''_1$};
            \node[label, gray] at (1,-0.3) {$e'_2$};
            \node[label, gray] at (3,-0.3) {$e''_2$};
		
	\end{tikzpicture}
	
	\caption{}\label{fig:doubleloopbar}
\end{figure}

\begin{proof}[Proof of Theorem \ref{mainresult}]
The ``if'' direction is clear since ascending HNN extensions of free abelian groups and (free abelian)-by-free groups are residually finite, see \cite{Hall49} and \cite{Baumslag71}, respectively.

Conversely, let $G$ be a residually finite rank $n$ GBS group with reduced labeled graph $\Gamma$.
First suppose that $\Gamma$ contains an edge-loop $e$ with labels $1,q$, where $q\neq1$. If $e$ is the only edge in $\Gamma$ then $G$ is an ascending HNN extension as in the statement.
Otherwise, $\Gamma$ contains a subgraph isomorphic to either the first graph in Figure \ref{fig:doubleloop} or the graph in Figure \ref{fig:loopandedge}. In this case, the subgraph corresponds to a non-residually-finite subgroup of $G$ by Lemmas \ref{lem:induceimm}, \ref{lemmadoubleloop} and \ref{lem:loopandedge}, contradicting residual finiteness of $G$.

Now suppose that $\Gamma$ has no edge-loop $e$ with labels $1,q$, where $q\neq 1$. 
Let $\bar{\Gamma}\to \Gamma$ be a covering of finite labeled graphs, and such that an edge of $\bar{\Gamma}$ is a loop if and only if both its labels equal $1$. By Lemma \ref{lem:induceimm}, we obtain a finite-index subgroup  $\bar{G}<G$ with associated labeled graph $\bar{\Gamma}$. If an edge $\bar{e}$ in $\bar{\Gamma}$ has a label of $1$ then it must map to an edge $e$ in $\Gamma$ with a label of $1$; then $e$ must be a loop since $\Gamma$ is reduced, and by our assumption on $\Gamma$, the other label of $e$ must also equal $1$; but then both labels on $\bar{e}$ are $1$, so $\bar{e}$ is a loop. Hence  $\bar{\Gamma}$ is reduced. Finally, we apply Lemma \ref{lemma3} to $\bar{G}$ to deduce that $\bar{G}$ is virtually $\mathbb{Z}^n$-by-free, hence $G$ is as well.

\end{proof}

\section{Subgroup separability}\label{subsep}

The aim of this section is to characterize subgroup separable rank $n$ GBS groups. For that, we use the following result about subgroup separability of finite graphs of groups.

\begin{thm}\cite[Theorem 2]{RaptisVarsos96}\label{thm:RaptisVarsos}
Let $G$ be the fundamental group of a finite graph of groups, where each vertex group $G_v$ is a finitely generated torsion free nilpotent group. Suppose that every edge group $G_e$ is of finite index in the corresponding vertex groups. The following propositions are equivalent:
\begin{itemize}
    \item[i)] The group $G$ is subgroup separable.
    \item[ii)] For each vertex group $G_v$ we have that $\bigcap_{x\in G} x G_v x^{-1}$ has finite index in $G_v$.
\end{itemize}
\end{thm}

\begin{thm}\label{thm:subgpsep}
Let $G$ be a rank $n$ GBS group. Then $G$ is subgroup separable if and only if $G$ is virtually $\mathbb{Z}^n$-by-free.
\end{thm}

\begin{proof}
Suppose that $G$ is subgroup separable. Then, $G$ is in particular residually finite, so by Theorem \ref{mainresult}, $G$ is either an ascending HNN extension with free abelian vertex group or virtually $\mathbb{Z}^n$-by-free. But by \cite[Proposition 1]{RaptisVarsos96} strictly ascending HNN extensions are not subgroup separable. Hence, $G$ has to be virtually $\mathbb{Z}^n$-by-free.

Since subgroup separability passes to finite extensions, it remains to show that if $G$ is $\mathbb{Z}^n$-by-free, then $G$ is subgroup separable. By Proposition \ref{prop:normalZn}, there is a normal subgroup $K$ that acts trivially on the Bass-Serre tree $T$ associated to $G$ and has finite index in every vertex group. In particular, the group $K$ is contained in $x G_v x^{-1}$ for each vertex group $G_v$ and $x\in G$, so condition ii) from Theorem \ref{thm:RaptisVarsos} is satisfied and we get that $G$ is indeed subgroup separable.
\end{proof}

\section{Cyclic subgroup separability of ascending HNN extensions}\label{sec:cycsubsep}

Note that since subgroup separability implies cyclic subgroup separability, and this further implies residual finiteness, from Theorem \ref{mainresult} and Theorem \ref{thm:subgpsep} we deduce that the only case that remains to consider is when $G$ is an ascending HNN extension. More concretely, in this section we determine when the following group is cyclic subgroup separable

\[G=A*_\varphi=\langle A,t\mid tat^{-1}=\varphi(a), a\in A\rangle,\]
where $A\cong\Z^n$ and $\varphi\colon A\to A$ is a monomorphism.
Later in this section we will analyze the map $\varphi$ from the perspective of linear algebra, so we will use additive notation for the subgroup $A$. However, we will use multiplicative notation when multiplying by an element that is not in $A$ (such as a power of the stable letter $t$).

We start with a lemma that characterizes finite-index normal subgroups of $G$.

\begin{lem}\label{lem:GbarK}
    \begin{enumerate}
        \item\label{item:barG} If $\bar{G}\triangleleft G$ is a finite-index normal subgroup, then $K:=\bar{G}\cap A$ is a finite-index subgroup of $A$ such that $\varphi(K)<K$ and $a-\varphi^r(a)\in K$ for all $a\in A$, where $t^r$ is the least positive power of $t$ contained in $\bar{G}$.
        \item\label{item:K} If $K< A$ is a finite-index subgroup and $0< r\in\Z$ such that $\varphi(K)<K$ and $a-\varphi^r(a)\in K$ for all $a\in A$, then $\bar{G}:=\langle K,t^r\rangle$ is a finite-index normal subgroup of $G$, and $\bar{G}\cap A=K$.
    \end{enumerate}
\end{lem}
\begin{proof}
    \begin{enumerate}
        \item Suppose $\bar{G}\triangleleft G$ is a finite-index normal subgroup and let $K:=\bar{G}\cap A$. Clearly, $K$ is a finite-index subgroup of $A$.
        It follows from the presentation of $G$ that $tKt^{-1}=\varphi(K)$, so $\varphi(K)<K$ by normality of $\bar{G}$.
        For $a\in A$, normality of $\bar{G}$ and the fact that $t^r\in\bar{G}$ implies that
        $$(a-\varphi^r(a))t^r=at^r (-a)\in\bar{G},$$
        so $a-\varphi^r(a)\in K$.

        \item Suppose $K< A$ is a finite-index subgroup and $0< r\in\Z$ such that $\varphi(K)<K$ and $a-\varphi^r(a)\in K$ for all $a\in A$. Let $\bar{G}:=\langle K,t^r\rangle$.
        Let us first show that $\bar{G}$ is normal in $G$.
        We have $tKt^{-1}=\varphi(K)<K$, so $t\bar{G}t^{-1}<\bar{G}$ and
        $$t^{-1}\bar{G}t=t^{-r}t^{r-1}\bar{G}t^{-(r-1)}t^r<t^{-r}\bar{G}t^r=\bar{G}.$$
        For $a\in A$ we have that $a+K-a=K$ and
        $$at^r (-a)=(a-\varphi^r(a))t^r\in Kt^r<\bar{G},$$
        hence $a\bar{G}(-a)<\bar{G}$. This proves that $\bar{G}$ is normal in $G$.
        To see that $\bar{G}$ has finite index in $G$, consider the quotient map $G\to G/\bar{G}$. We can write $G$ as a product of subgroups $\langle t^{-1}\rangle A \langle t\rangle$ and each of these subgroups clearly has finite image in $G/\bar{G}$, hence so does $G$.
        
        Finally, we can apply the relations $t^ra=\varphi^r(a)t^r$ and $at^{-r}=t^{-r}\varphi^r(a)$ for $a\in A$, and use the fact that $\varphi^r(K)<K$, to write any element $g\in\bar{G}$ in the form $g=t^{-ir}kt^{jr}$ with $i,j\geq0$ and $k\in K$.
        By normal forms, we have that $g\in A$ if and only if $i=j$ and $k\in\varphi^{ir}(A)$.
        In this case, $g=a\in A$ and $\varphi^{ir}(a)=k$, so
        $$g=a-\varphi^r(a)+\varphi^r(a)-\varphi^{2r}(a)+\varphi^{2r}(a)-\dots-\varphi^{ir}(a)+\varphi^{ir}(a),$$
        which is an element of $K$ by our assumptions on $K$. Therefore $\bar{G}\cap A=K$.     
    \end{enumerate}
\end{proof}

\begin{rem}\label{rem:avarphi}
In both cases of Lemma \ref{lem:GbarK} we have that $a-\varphi^{xr}(a)\in K$ for all $a\in A$ and $x\geq 1$, which can also be written $a\in\varphi^{xr}(a)+K$. Indeed, the normality of $\bar{G}$ and the fact that $t^{xr}\in\bar{G}$ implies that
$$(a-\varphi^{xr}(a))t^{xr}=at^{xr} (-a)\in\bar{G},$$
        so $a-\varphi^{xr}(a)\in K$.
        As a consequence, we have that $\varphi^x(A)+K=A$ for all $x\geq1$.
\end{rem}

For $g_1,g_2\in G$, we say that $\langle g_1\rangle$ is \emph{separable from} $g_2$ if there is a finite-index normal subgroup $\bar{G}\triangleleft G$ with $g_2\notin\langle g_1\rangle\bar{G}$.
Clearly $G$ is cyclic subgroup separable if and only if $\langle g_1\rangle$ is separable from $g_2$ for all $g_1,g_2\in G$ with $g_2\notin\langle g_1\rangle$.
Over the course of the next three lemmas we reduce to the case where $g_1,g_2\in A$.

\begin{lem}\label{lem:barGcapA}
    Let $g=at^i$ with $a\in A$ and $i>0$, and let $1\neq b\in A$.
    Then there exists a finite-index normal subgroup $\bar{G}\triangleleft G$ such that $\langle g\rangle\bar{G}\cap A=\bar{G}\cap A$ and $b\notin \bar{G}$.
\end{lem}
\begin{proof}
    The powers of $g$ are given by $g^m=a_m t^{im}$ ($m\geq1$), where
        \begin{equation}\label{am}
            a_m=a+\varphi^i(a)+\varphi^{2i}(a)+\cdots+\varphi^{(m-1)i}(a).
        \end{equation}
        Theorem \ref{mainresult} tells us that $G$ is residually finite, so there is a finite-index normal subgroup $\bar{G}\triangleleft G$ not containing $b$.
        Lemma \ref{lem:GbarK}\ref{item:barG} tells us that $K:=\bar{G}\cap A$ is a finite-index subgroup of $A$ such that $\varphi(K)<K$. Hence $\varphi$ descends to a well-defined homomorphism $\bar{\varphi}\colon A/K\to A/K$ given by $a'+K\mapsto \varphi(a')+K$. Remark \ref{rem:avarphi} implies that $\bar{\varphi}$ is surjective, and $A/K$ is finite, so in fact $\bar{\varphi}$ is an automorphism of $A/K$. Let $l\geq1$ be such that $\bar{\varphi}^l$ is the identity map on $A/K$, and let $a_l+K$ have order $q$ in $A/K$.
        Then for any $x\geq1$ we have
        \begin{align*}
            a_{lqx}+K&=a_l+\varphi^{li}(a_l)+\varphi^{2li}(a_l)+\cdots+\varphi^{(qx-1)li}(a_l)+K\\
            &=qxa_l+K\\
            &=K.
        \end{align*}

        Lemma \ref{lem:GbarK}\ref{item:barG} also tells us that $a'-\varphi^r(a')\in K$ for all $a'\in A$, where $t^r$ is the least positive power of $t$ contained in $\bar{G}$. 
        And by Remark \ref{rem:avarphi}, $a'-\varphi^{xr}(a')\in K$ for all $a'\in A$ and $x\geq 1$.        
        Then by Lemma \ref{lem:GbarK}\ref{item:K}, $\bar{G}':=\langle K,t^{ilqr}\rangle$ is a finite-index normal subgroup of $G$ with $\bar{G}^\prime \cap A=K$. 
        Note that $g^m=a_m t^{im}$ is contained in $A\bar{G}'$ if and only if $lqr$ divides $m$, so
        \begin{align*}
            \langle g\rangle \bar{G}'\cap A &= \langle g^{lqr}\rangle \bar{G}'\cap A\\
            &=\langle a_{lqr} t^{ilqr}\rangle \bar{G}'\cap A\\
            &=\langle a_{lqr}\rangle \bar{G}'\cap A\\
            &=\langle a_{lqr}\rangle+ K\\
            &=K.
        \end{align*}
        We also have $b\notin\bar{G}'$ since $b\notin K$.
\end{proof}

\begin{lem}\label{lem:notAconj}
    Every cyclic subgroup of $G$ that is not conjugate into $A$ is separable.
\end{lem}
\begin{proof}
    We must show that $\langle g_1 \rangle$ is separable from $g_2$ whenever $g_1,g_2\in G$, with $g_1$ not conjugate into $A$ and $g_2\notin\langle g_1\rangle$. We do this in four steps.
    Throughout, we will use the fact that any element of $G$ can be written in the form $t^{-i}at^j$ for $a\in A$ and $i,j\geq 0$, and such an element is conjugate into $A$ if and only if $i=j$.
    We will also use the fact that, for any $h\in G$, $\langle g_1 \rangle$ is separable from $g_2$ if and only if $\langle hg_1h^{-1}\rangle$ is separable from $hg_2h^{-1}$.
    \begin{enumerate}

        \item\label{item:g1notAg2A} Suppose $g_1=at^i$ for $a\in A$ and $i>0$, and $g_2\in A$. 
        By Lemma \ref{lem:barGcapA}, there exists a finite-index normal subgroup $\bar{G}\triangleleft G$ such that $\langle g_1\rangle\bar{G}\cap A=\bar{G}\cap A$ and $g_2\notin \bar{G}$. 
        It follows that $g_2\notin\langle g_1\rangle\bar{G}$, as required.

        \item\label{item:g1g2notA} Suppose $g_1=at^i$ for $a\in A$ and $i>0$, and $g_2=bt^j$ for $b\in A$ and $j>0$.
        If $i$ does not divide $j$, then we can separate $\langle g_1\rangle$ from $g_2$ using $\bar{G}=\langle A,t^i\rangle$.

        Now suppose $i$ does divide $j$, say $j=li$.
        We have $g_1^l=a_l t^{il}=a_l t^j$, with $a_l$ defined as in (\ref{am}).
        By assumption, $g_1^l\neq g_2$, so $1\neq g_2 g_1^{-l}=b -a_l\in A$.
        By Lemma \ref{lem:barGcapA}, there exists a finite-index normal subgroup $\bar{G}\triangleleft G$ such that $\langle g_1\rangle\bar{G}\cap A=\bar{G}\cap A$ and $g_2 g_1^{-l}\notin \bar{G}$.
        We claim that $g_2\notin\langle g_1\rangle\bar{G}$.
        Indeed, if $g_2\in g_1^q\bar{G}$ for some $q\in\Z$, then $g_2 g_1^{-l}\in g_1^{q-l}\bar{G}\cap A\subseteq \bar{G}\cap A$, contradicting $g_2 g_1^{-l} \notin\bar{G}$.

        \item\label{item:g1ipos} Suppose $g_1=at^i$ for $a\in A$ and $i>0$. Write $g_2=t^{-j}bt^k$ for $b\in A$ and $j,k\geq0$. Replacing $g_2$ with $g_2^{-1}$ if necessary, we may assume that $j\leq k$.
        It suffices to show that $\langle t^j g_1 t^{-j}\rangle=\langle \varphi^j(a)t^i\rangle$ is separable from $t^j g_2 t^{-j}=bt^{k-j}$, and this follows from parts  \ref{item:g1notAg2A} and \ref{item:g1g2notA}. 

        \item Finally, we deal with the general case. Write $g_1=t^{-i}at^j$ for $a\in A$ and $i,j\geq0$. Replacing $g_1$ with $g_1^{-1}$ if necessary, we may assume that $i< j$ (we cannot have $i=j$ since $g_1$ is not conjugate into $A$).
        It suffices to show that $\langle t^i g_1 t^{-i}\rangle=\langle at^{j-i}\rangle$ is separable from $t^i g_2 t^{-i}$, and this follows from part \ref{item:g1ipos}.\qedhere
    \end{enumerate}
\end{proof}

\begin{lem}\label{lem:cyclicsubsep}
    If $\langle g_1 \rangle$ is separable from $g_2$ whenever $g_1,g_2\in A$ and $g_2\notin\langle g_1\rangle$, then $G$ is cyclic subgroup separable.
\end{lem}
\begin{proof}
By Lemma \ref{lem:notAconj}, it suffices to show that $\langle g_1\rangle$ is separable for all $g_1\in A$.
Let $g_2\in G-\langle g_1\rangle$. We must show that $\langle g_1\rangle$ is separable from $g_2$. Write $g_2=t^{-i}at^j$ for $a\in A$ and $i,j\geq 0$.
        It suffices to show that $\langle t^i g_1 t^{-i}\rangle=\langle\varphi^i(g_1)\rangle$ is separable from $t^ig_2t^{-i}=at^{j-i}$. If $j-i=0$, then this is possible by assumption of the lemma, otherwise we can separate using $\bar{G}=\langle A, t^{|j-i|+1}\rangle$ (noting that $\langle\varphi^i(g_1)\rangle<\bar{G}$ and $at^{j-i}\notin\bar{G}$).
\end{proof}

Over the next four lemmas we show that cyclic subgroup separability of $G$ is equivalent to a statement about the orders of elements of $A$ in finite quotients of $G$.

From now on let $d=|A:\varphi(A)|$. For $m\in\N$, let $mA$ denote the subgroup $\{ma\mid a\in A\}<A$.
For any $a\in A$ we have $\varphi(ma)=m\varphi(a)\in mA$, so $\varphi(mA)< mA$.
Therefore, $\varphi$ descends to a well-defined homomorphism $\varphi_m\colon A/mA\to A/mA$ given by $a+mA\mapsto\varphi(a)+mA$.

\begin{lem}\label{lem:Am}
    If $m\in\N$ is coprime to $d$, then there exists a finite-index normal subgroup $\bar{G}\triangleleft G$ with $\bar{G}\cap A=mA$.
\end{lem}
\begin{proof}
    By Lemma \ref{lem:GbarK}\ref{item:K}, it suffices to find $r\geq 1$ such that $a-\varphi^r(a)\in mA$ for all $a\in A$, or equivalently $\varphi^r(a)+mA= a+mA$ for all $a\in A$.
    Clearly $|A:mA|=m^n$, so $|A: \varphi(A)+mA|$ divides $m^n$, but it also divides $d=|A:\varphi(A)|$. Since $m$ is coprime to $d$ we have $\varphi(A)+mA=A$. Therefore, the homomorphism $\varphi_m\colon A/mA\to A/mA$ defined above is an automorphism. Let $r\geq1$ be such that $\varphi_m^r$ is the identity map on $A/mA$. Then $\varphi^r(a)+mA= a+mA$ for every $a\in A$, as required.
\end{proof}

\begin{lem}\label{lem:trivintersect}
    If $0\neq g_1, g_2\in A$ with $\langle g_1\rangle\cap\langle g_2\rangle=\{0\}$, then $\langle g_1 \rangle$ is separable from $g_2$.
\end{lem}
\begin{proof}
It follows from $\langle g_1 \rangle \cap \langle g_2 \rangle= \{0\}$ that $\langle g_1,g_2\rangle$ is a free abelian subgroup of rank two in $A$. Then, by the basis extension property for free abelian groups, there are two subgroups $A_1, A_2$ of $A$ such that $A_1 \cong \mathbb{Z}^2$, $A_2 \cong \mathbb{Z}^{n-2}$, $\langle g_1, g_2 \rangle$ has finite index in $A_1$ and $A= A_1 \oplus A_2$. Hence, if we denote by $\tilde{A}$ the subgroup $\langle g_1 , g_2 \rangle \oplus A_2$, then $\tilde{A}$ has finite index in $A$, say it has index $m$.

Now let $p$ be a prime that does not divide $d$ or $m$. Since $|A:pA|=p^n$, the index $|A: \tilde{A} +pA|$ must be a power of $p$, but it also divides $m$, so $|A: \tilde{A}+pA|=1$ and $A= \tilde{A} +pA$.

Let $A_2$ be generated by $\{g_3, g_4,\dots,g_n\}$.
We have $A=\langle g_1, g_2,\dots,g_n\rangle +pA$, and $A/pA$ cannot be generated by fewer than $n$ elements, therefore $g_2 \notin \langle g_1 \rangle +pA$.
By Lemma \ref{lem:Am}, there exists a finite-index normal subgroup $\bar{G}\triangleleft G$ with $\bar{G}\cap A=pA$. It follows that $g_2\notin\langle g_1\rangle\bar{G}$, as required.
\end{proof}

\begin{lem}\label{lem:ordermult}
    If $H$ is any cyclic subgroup separable group, then for any $1\neq h\in H$ and any $m\in\N$, there exists a finite-index normal subgroup $\bar{H}\triangleleft H$ such that the order of $h$ in the quotient $H/\bar{H}$ is a multiple of $m$.
\end{lem}
\begin{proof}
    Take a finite-index normal subgroup $\bar{H}\triangleleft H$ such that $h,h^2,\dots,h^{m-1}\notin\langle h^{m}\rangle\bar{H}$. Then the order of $h$ in the quotient $H/\bar{H}$ is a multiple of $m$. Indeed, otherwise there are $x,y\in\Z$ with $1\leq y<m$ such that $h^{xm+y}\in\bar{H}$, and then $h^y\in\langle h^{m}\rangle\bar{H}$, contrary to our assumption on $\bar{H}$.
\end{proof}

\begin{lem}\label{lem:ordermultiple}
    $G$ is cyclic subgroup separable if and only if, for all $0\neq a\in A$, for all primes $p$, and for all $s\in\N$, there exists a finite-index normal subgroup $\bar{G}\triangleleft G$ such that the order of $a$ in the quotient $G/\bar{G}$ is a multiple of $p^s$.
\end{lem}
\begin{proof}
If $G$ is cyclic subgroup separable then the claim follows from Lemma \ref{lem:ordermult}.
    Conversely, suppose that for all $0\neq a\in A$, for all primes $p$, and for all $s\in\N$, there exists a finite-index normal subgroup $\bar{G}\triangleleft G$ such that the order of $a$ in the quotient $G/\bar{G}$ is a multiple of $p^s$.
    By Lemma \ref{lem:cyclicsubsep} and Lemma \ref{lem:trivintersect}, to show that $G$ is cyclic subgroup separable it suffices to show that $\langle g_1 \rangle$ is separable from $g_2$ when $0\neq g_1,g_2\in A$ have a common nonzero multiple and $g_2\notin\langle g_1\rangle$. Such $g_1$, $g_2$ can be written as $g_1=xa$ and $g_2=ya$ with $0\neq a\in A$ and $x,y\in\Z$ with $x$ not dividing $y$. In this case there must be a prime power $p^s$ that divides $x$ but not $y$.
    Then by assumption we may take a finite-index normal subgroup $\bar{G}\triangleleft G$ such that the order of $a$ in the quotient $G/\bar{G}$ is a multiple of $p^s$ -- say $a$ has order $qp^s$ in the quotient.
    Then $qg\in\bar{G}$ for all $g\in\langle g_1\rangle$ but $qg_2\notin\bar{G}$, so $g_2\notin\langle g_1\rangle\bar{G}$.
\end{proof}

Now we turn to some linear algebra.
Given a basis of $A$, we can write $\varphi$ as an integer matrix.
From this, we can define the characteristic equation $f_\varphi(x)=0$, the trace tr$(\varphi)$ and the determinant det$(\varphi)$ of $\varphi$. These are all independent of the choice of basis.
By considering the Smith normal form of $\varphi$, it is clear that $|\det(\varphi)|=d$.
If $\lambda\in\Z$ is an eigenvalue of $\varphi$ (i.e. an integer root of the characteristic equation) then $\varphi(a)=\lambda a$ for some $0\neq a\in A$. Such an $a$ is called an \emph{eigenvector} of $\varphi$.
Non-integer eigenvalues are irrational so do not play a role in our setting.
The following lemma shows that integer eigenvalues with absolute value greater than 1 provide an easy obstruction to cyclic subgroup separability.

\begin{lem}\label{lem:eigenvalues'}
    If $\varphi$ has an eigenvalue $\lambda\in\Z$ with $|\lambda|>1$, and if $0\neq a\in A$ is a corresponding eigenvector, then $\langle \lambda a\rangle$ is not separable from $a$.
\end{lem}
\begin{proof}
Suppose for contradiction that $\langle \lambda a\rangle$ is separable from $a$. Then there is a finite-index normal subgroup $\bar{G}\triangleleft G$ with $a\notin\langle \lambda a\rangle\bar{G}$. In particular, $a\bar{G}$ has a larger order than $(\lambda a)\bar{G}$ in the finite quotient $G/\bar{G}$. But $tat^{-1}=\varphi(a)=\lambda a$, so $a\bar{G}$ and $(\lambda a)\bar{G}$ are conjugate in $G/\bar{G}$ and they must have the same order, a contradiction.    
\end{proof}

\begin{rem}
Another way to see why $G$ is not cyclic subgroup separable in this case is to observe that the subgroup $\langle a,t\rangle<G$ is isomorphic to the Baumslag--Solitar group BS(1,$\lambda$).
\end{rem}

Viewing $A=\Z^n<\Q^n$, let $\varphi_\Q$ denote the $\Q$-linear map $\Q^n\to\Q^n$ induced by $\varphi$.
Let 
$$\{0\}=V_0\leq V_1\leq V_2\leq \dots \leq V_l=\Q^n$$
be a composition series for $\Q^n$ as a $\Q[\varphi_\Q]$-module, i.e. the $V_i$ are distinct $\varphi_\Q$-invariant subspaces of $\Q^n$ such that there is no $\varphi_\Q$-invariant subspace $V_i\lneq V\lneq V_{i+1}$ for any $0\leq i<l$.
Let $A_i=V_i\cap A$.
Then
$$\{0\}=A_0< A_1< A_2< \dots < A_l=A$$
are $\varphi$-invariant subgroups of $A$, such that any $\varphi$-invariant subgroup $A_i\lneq B< A_{i+1}$ ($0\leq i<l$) has finite index in $A_{i+1}$ (indeed otherwise the $\Q$-span of $B$ would yield a $\varphi_\Q$-invariant subspace strictly between $V_i$ and $V_{i+1}$).

For $1\leq i\leq l$, we note that $A_i/A_{i-1}$ is torsion-free, say it is isomorphic to $\Z^{n_i}$. Let $\varphi_i$ denote the endomorphism of $A_i/A_{i-1}$ induced by $\varphi$, and let $f_i(x)$ be its characteristic polynomial.
Note that $\varphi$ induces a $\Q$-linear map of $V_i/V_{i-1}$ with the same characteristic polynomial, and $f_i(x)$ is irreducible over $\Q$ by the following lemma.

\begin{lem}
    Let $\alpha$ be a linear map of a finite-dimensional vector space $V$ over a field $F$.
    Then the characteristic polynomial $f_\alpha(x)$ is irreducible over $F$ if and only if there is no proper non-trivial $\alpha$-invariant subspace of $V$.    
\end{lem}
\begin{proof}
    If $W\leq V$ is a proper non-trivial $\alpha$-invariant subspace, then the characteristic equation of $\alpha|_W$ divides $f_\alpha(x)$ (this is clear if you extend a basis of $W$ to a basis of $V$, and write $\alpha$ as a matrix with respect to this basis), so $f_\alpha(x)$ is reducible over $F$.

    Conversely, if there is no proper non-trivial $\alpha$-invariant subspace of $V$, then we deduce from the rational canonical form of $\alpha$ that the minimal polynomial of $\alpha$ is the same as the characteristic polynomial $f_\alpha(x)$. Now suppose for contradiction that there is a factorization $f_\alpha(x)=q(x)r(x)$ with $q(x)$, $r(x)$ polynomials over $F$ of smaller degree than $f_\alpha(x)$. The minimal polynomial $f_\alpha(x)$ does not divide $q(x)$, so the kernel $U$ of $q(\alpha)$ is a proper $\alpha$-invariant subspace of $V$. Similarly, $f_\alpha(x)$ does not divide $r(x)$, so the image of $r(\alpha)$ is a non-trivial subspace of $V$, and it is contained in $U$ since $0=f_\alpha(\alpha)=q(\alpha)r(\alpha)$ by the Cayley--Hamilton Theorem. Hence $U$ is a proper non-trivial $\alpha$-invariant subspace of $V$, a contradiction.
\end{proof}

If we choose a basis of $A$ that restricts to give bases of the $A_i$, then $\varphi$ is represented by a block upper triangular matrix with respect to this basis, where the blocks correspond to the endomorphisms $\varphi_i$.
It follows that
$$f_\varphi(x)=f_1(x)f_2(x)\dots f_l(x)$$
is the unique factorization of $f_\varphi(x)$ into monic irreducible polynomials over $\Q$.

The key to determining cyclic subgroup separability of $G$ will be the relationship between the subgroups $A_i$ and the following family of subgroups.
For $p$ a prime, $0\leq i<l$ and $m\in\N$, define
$$K_{p^m,i}:=\{a\in A\mid \varphi^j(a)\in A_i+p^m A\text{ for some }j\geq0\}.$$

\begin{lem}\label{lem:ppower}
    $|A:K_{p^m,i}|$ is a power of $p$.
\end{lem}
\begin{proof}
    This follows from the fact that $p^m A< K_{p^m,i}$ and $|A:p^m A|=p^{mn}$.
\end{proof}

\begin{lem}\label{lem:Kp^mKp}
    If $p^{m-1}a\in K_{p^m,i}$ with $a\in A$, then $a\in K_{p,i}$.
\end{lem}
\begin{proof}
    We have $p^{m-1}\varphi^j(a)=\varphi^j(p^{m-1}a)\in A_i+p^mb$ for some $j\geq0$ and some $b\in A$.
    Hence $p^{m-1}(\varphi^j(a)-pb)\in A_i$, and $\varphi^j(a)-pb\in A_i$ since $A_i$ is a direct factor of $A$.
    Then $\varphi^j(a)\in A_i+pA$, so $a\in K_{p,i}$.
\end{proof}

\begin{lem}\label{lem:Kp^mivarphiinv}
    $\varphi(K_{p^m,i})< K_{p^m,i}$.
\end{lem}
\begin{proof}
    If $a\in K_{p^m,i}$ with $\varphi^j(a)\in A_i+p^m A$, then $\varphi^j(\varphi(a))\in\varphi(A_i)+\varphi(p^m A)< A_i+p^m A$, so $\varphi(a)\in K_{p^m,i}$.
\end{proof}

\begin{lem}\label{lem:existsbarG}
    There exists a finite-index normal subgroup $\bar{G}\triangleleft G$ with $\bar{G}\cap A=K_{p^m,i}$.
\end{lem}
\begin{proof}
Lemma \ref{lem:Kp^mivarphiinv} implies that $\varphi$ induces a well-defined endomorphism of $A/K_{p^m,i}$.
It is clear that the preimage $\varphi^{-1}(K_{p^m,i})$ is equal to $K_{p^m,i}$, so $\varphi$ induces an injective endomorphism of $A/K_{p^m,i}$.
    But $A/K_{p^m,i}$ is finite, so $\varphi$ in fact induces an automorphism, and there is some $r\geq1$ such that $\varphi^r$ induces the identity.
    This means that $a-\varphi^r(a)\in K_{p^m,i}$ for all $a\in A$.
    It then follows from Lemma \ref{lem:GbarK} that $\bar{G}=\langle K_{p^m,i}, t^r\rangle$ is a finite-index normal subgroup of $G$ with $\bar{G}\cap A=K_{p^m,i}$.
\end{proof}

We are now ready to state and prove the main theorem of the section.

\begin{thm}\label{thm:css}
    The following are equivalent.
    \begin{enumerate}
        \item\label{item:cycsubsep} $G$ is cyclic subgroup separable.
        \item\label{item:Kpi} There is no prime $p$ and $0\leq i<l$ with $A_{i+1}< K_{p,i}$.
        \item\label{item:equivx^ni} There is no prime $p$ and $1\leq i\leq l$ with $f_i(x)\equiv x^{n_i}$ {\rm{(mod $p$)}}.
    \end{enumerate}
\end{thm}
\begin{proof}
    \ref{item:cycsubsep} $\Rightarrow$ \ref{item:Kpi}:
    Suppose for contradiction that \ref{item:cycsubsep} holds but not \ref{item:Kpi}, and fix $p,i$ with $A_{i+1}< K_{p,i}$.
    Take $a\in A_{i+1}-A_i$.
    By Lemma \ref{lem:ordermultiple}, there is a finite-index normal subgroup $\bar{G}\triangleleft G$ such that the order of $a$ in $G/\bar{G}$ is a multiple of $p$.
    Let $K=\bar{G}\cap A$.

    By choosing an appropriate basis of $K$, we can write $K$ as a direct sum $K=K_i\oplus K_i'$, where $K_i=K\cap A_i$ and $\langle a\rangle\cap K< K_i'$.
    This direct sum extends to a direct sum $A=A_i\oplus A_i'$, where $A_i'$ consists of those elements of $A$ with nonzero multiples in $K_i'$.
    As $\langle a\rangle < A_i'$ and $K+A_i=A_i\oplus K_i'$, the order of $a$ in $A/(K+A_i)$ is the same as the order of $a$ in $A/K$, which is a multiple of $p$ by assumption.
    By considering the structure of the finite abelian group $A/(K+A_i)$, it is not hard to show that there exists a subgroup $K+A_i< L< A$ with $a\notin L$ and $|A:L|=p^m$ for some $m\geq1$.
    It follows that $p^m A< L$ and 
    \begin{equation}\label{anotin}
    a\notin K+A_i+p^m A.
    \end{equation}

    Next, we claim that $A_{i+1}< K_{p^s,i}$ for all $s\geq1$.
    We proceed by induction on $s$.
    The case $s=1$ holds by assumption.
    For $s>1$ and $a_{i+1}\in A_{i+1}$, we have $\varphi^j(a_{i+1})\in A_i+p^{s-1}A$ for some $j\geq0$ by the induction hypothesis, say $\varphi^j(a_{i+1})=a_i+p^{s-1}a'$ for some $a_i\in A_i$ and $a'\in A$.
    But then $a'\in A_{i+1}< K_{p,i}$, so $\varphi^{j'}(a')\in A_i+pA$ for some $j'\geq0$.
    Therefore $\varphi^{j+j'}(a_{i+1})\in A_i+p^sA$, and $a_{i+1}\in K_{p^s,i}$ as required.

    The above claim with $s=m$ tells us that $\varphi^j(a)\in A_i+p^mA$ for some $j\geq0$.
    But Lemma \ref{lem:GbarK} implies that $a-\varphi^r(a)\in K$ for some $r\geq1$, and Remark \ref{rem:avarphi} yields $a-\varphi^{sr}(a)\in K$ for all $s\geq1$.
    Choosing $s$ such that $sr\geq j$, we see that 
    $$a\in K+\varphi^{sr}(a)< K+A_i+p^m A,$$
    contradicting (\ref{anotin}).

    \ref{item:Kpi} $\Rightarrow$ \ref{item:cycsubsep}:
    Suppose for contradiction that \ref{item:Kpi} holds but not \ref{item:cycsubsep}.
    By Lemma \ref{lem:ordermultiple}, there is $0\neq a\in A$, a prime $p$ and $s\geq1$ such that, for any finite-index normal subgroup $\bar{G}\triangleleft G$, the order of $a$ in $G/\bar{G}$ is not a multiple of $p^s$.
    Suppose $a\in A_{i+1}-A_i$ ($0\leq i<l$).
    By Lemma \ref{lem:existsbarG}, for each $m\geq1$ there exists a finite-index normal subgroup $\bar{G}_m\triangleleft G$ with $\bar{G}_m\cap A=K_{p^m,i}$.
    The quotient $A/K_{p^m,i}$ is a finite $p$-group by Lemma \ref{lem:ppower}, so the order of $a$ in $A/K_{p^m,i}$ is a power of $p$; but the order is not a multiple of $p^s$ by assumption, hence it divides $p^{s-1}$. In particular $p^sa\in K_{p^m,i}$.

    We know from \ref{item:Kpi} that there is $b\in A_{i+1}-K_{p,i}$, and so $p^{m-1}b\notin K_{p^m,i}$ by Lemma \ref{lem:Kp^mKp}. This implies that the order of $b$ in $A/K_{p^m,i}$ is a multiple of $p^m$, so in particular $|A_{i+1}:A_{i+1}\cap K_{p^m,i}|\geq p^m$. As a result,
    $$K_{i+1}:=A_{i+1}\cap\bigcap_{m\geq1}K_{p^m,i}$$
    is an infinite-index subgroup of $A_{i+1}$.
    Each of the $K_{p^m,i}$ is $\varphi$-invariant by Lemma \ref{lem:Kp^mivarphiinv} and contains $p^sa$, therefore $K_{i+1}$ is a $\varphi$-invariant infinite-index subgroup of $A_{i+1}$ that strictly contains $A_i$, which contradicts the construction of the groups $A_0< A_1 < \dots < A_l$.

    \ref{item:Kpi} $\Rightarrow$ \ref{item:equivx^ni}:
    Let $p$ be a prime and let $1\leq i\leq l$.
    We defined $f_i(x)$ to be the characteristic polynomial of the endomorphism $\varphi_i$ of $A_i/A_{i-1}$, so the reduction of $f_i(x)$ (mod $p$) is the characteristic polynomial of the induced $\Z/p\Z$-linear map $A_i/(A_{i-1}+pA_i)\to A_i/(A_{i-1}+pA_i)$.
    If $f_i(x)\equiv x^{n_i}$ {\rm{(mod $p$)}} then the Cayley--Hamilton Theorem implies that $\varphi^{n_i}(A_i)< A_{i-1}+pA_i$, which means that $A_i< K_{p,i-1}$, contradicting \ref{item:Kpi}.

    \ref{item:equivx^ni} $\Rightarrow$ \ref{item:Kpi}:
    Suppose for contradiction that \ref{item:equivx^ni} holds but not \ref{item:Kpi}.
    Let $p$ be a prime and let $0\leq i< l$ with $A_{i+1}< K_{p,i}$.
    Then for every $a\in A_{i+1}$ we have $\varphi^j(a)\in A_i+pA_{i+1}$ for some $j\geq0$. Since $A_{i+1}$ is finitely generated, there exists $j\geq0$ such that $\varphi^j(A_{i+1})< A_i+pA_{i+1}$.
    Let $\bar{\varphi}_{i+1}$ denote the $\Z/p\Z$-linear map $A_{i+1}/(A_i+pA_{i+1})\to A_{i+1}/(A_i+pA_{i+1})$ induced by $\varphi$.
    Since $\bar{\varphi}_{i+1}^j$ is the zero map, the minimal polynomial of $\bar{\varphi}_{i+1}$ must be a power of $x$.
    Every irreducible factor of the characteristic polynomial divides the minimal polynomial, so the characteristic polynomial of $\bar{\varphi}_{i+1}$ is also a power of $x$, i.e. $f_{i+1}(x)\equiv x^{n_{i+1}}$ {\rm{(mod $p$)}}, which contradicts \ref{item:equivx^ni}.    
\end{proof}

\begin{rem}\label{rem:exhibitsubgroup}
    If $f_i(x)\equiv x^{n_i}$ {\rm{(mod $p$)}} then there is no $a\in A_i-A_{i-1}$ with $\langle pa\rangle$ separable in $G$.
    Indeed, if $\langle pa\rangle$ was separable for some $a\in A_i-A_{i-1}$, then there would exist a finite-index normal subgroup $\bar{G}\triangleleft G$ such that the order of $a$ in $G/\bar{G}$ is a multiple of $p$ (see the proof of Lemma \ref{lem:ordermult}). Running the proof of \ref{item:Kpi} $\Rightarrow$ \ref{item:equivx^ni} above for our particular choices of $i$ and $p$ would tell us that $A_i< K_{p,i-1}$. Then running the proof of \ref{item:cycsubsep} $\Rightarrow$ \ref{item:Kpi} for our particular choice of $a$ and $\bar{G}$ (and the index $i$ shifted by 1) would yield a contradiction. 
\end{rem}

\begin{rem}
    Note that any integer eigenvalue $\lambda$ of $\varphi$ will give rise to a factor $f_i(x)=x-\lambda$ of the characteristic polynomial.
    If $|\lambda|>1$ then $\lambda$ has a prime factor $p$, and $f_i(x)\equiv x$ {\rm{(mod $p$)}}, so Theorem \ref{thm:css} implies that $G$ is not cyclic subgroup separable.
    This agrees with Lemma \ref{lem:eigenvalues'}.
\end{rem}

\begin{rem}\label{rem:varphisurj}
    If $\varphi \colon A\to A$ is surjective, then $d=|A:\varphi(A)|=|\det(\varphi)|=1$. As a result, every factor $f_i(x)$ of the characteristic polynomial has constant term equal to $\pm1$, so Theorem \ref{thm:css} implies that $G$ is cyclic subgroup separable. This is consistent with Theorem \ref{thm:subgpsep}, which tells us that $G$ is actually subgroup separable in this case.
\end{rem}

\begin{cor}\label{cor:n=2}
    Suppose $A\cong\Z^2$. Then $G$ is cyclic subgroup separable if and only if $\varphi$ has no integer eigenvalue $\lambda$ with $|\lambda|>1$ and $\tr(\varphi)$ is coprime to $d$.
\end{cor}
\begin{proof}
Since $\varphi$ corresponds to a $2\times2$ matrix, the characteristic equation of $\varphi$ takes the following form:
\begin{equation}\label{chareq}
    f_\varphi(x)=x^2-\tr(\varphi) x+\det(\varphi)=0.
\end{equation}
    If $\varphi$ has an integer eigenvalue $\lambda$ with $|\lambda|>1$ then Lemma \ref{lem:eigenvalues'} implies that $G$ is not cyclic subgroup separable.
    If $f_\varphi(x)$ is irreducible over $\Q$, then it follows from Theorem \ref{thm:css} and (\ref{chareq}) that $G$ is cyclic subgroup separable if and only if $\tr(\varphi)$ is coprime to $d=|\det(\varphi)|$.
    The remaining case is where $f_\varphi(x)=(x-\lambda_1)(x-\lambda_2)$ with $|\lambda_1|=|\lambda_2|=1$, so $d=1$ and $G$ is cyclic subgroup separable by Remark \ref{rem:varphisurj}.
\end{proof}

We finish with some examples for the group $G$, in each case determining whether it is cyclic subgroup separable.

\begin{exmp}\label{firstexample}
Let $G_1=\langle a,b,t \mid [a,b]=1, tat^{-1}=b^{-2}, tbt^{-1}= a b^{-3} \rangle$. 
This corresponds to $A=\langle a,b\rangle$, $\varphi(a)=b^{-2}$ and $\varphi(b)=ab^{-3}$.
Then the matrix corresponding to $\varphi$ (with respect to the basis $\{a,b\}$ of $A$) is
\begin{equation*}
C_1=
\begin{pmatrix}
0 & 1 \\
-2 & -3
\end{pmatrix}
.
\end{equation*}
This matrix has eigenvalues $\lambda_1= -1$ and $\lambda_2= -2$. The vectors $v_1= (1,-1)$ and $v_2= (1,-2)$ are eigenvectors associated to $\lambda_1$ and $\lambda_2$, respectively. From Lemma \ref{lem:eigenvalues'} we get that $\langle (ab^{-2})^2 \rangle$ is not separable from $ab^{-2}$, in particular $G_1$ is not cyclic subgroup separable.
\end{exmp}

\begin{exmp}
Let $G_2=\langle a,b,t \mid [a,b]=1, tat^{-1}= ab^{-1}, tbt^{-1}=ab \rangle$. This corresponds to $A=\langle a,b\rangle$, $\varphi(a)=ab^{-1}$ and $\varphi(b)=ab$. The integer matrix corresponding to $\varphi$ (with respect to the basis $\{a,b\}$ of $A$) is
\begin{equation*}
C_2=
\begin{pmatrix}
1 & 1 \\
-1 & 1
\end{pmatrix}
,
\end{equation*}
and this matrix does not even have real eigenvalues. The determinant and the trace of the matrix are both $2$, so from Corollary \ref{cor:n=2} we deduce that $G_2$ is not cyclic subgroup separable.

\end{exmp}

\begin{exmp}
Let $G_3=\langle a,b,t \mid [a,b]=1, tat^{-1}= ab^{2}, tbt^{-1}=a^2b^2 \rangle$. This corresponds to $A=\langle a,b\rangle$, $\varphi(a)=ab^{2}$ and $\varphi(b)=a^2b^2$. The integer matrix corresponding to $\varphi$ (with respect to the basis $\{a,b\}$ of $A$) is
\begin{equation*}
C_3=
\begin{pmatrix}
1 & 2 \\
2 & 2
\end{pmatrix}
,
\end{equation*}
which does not have integer eigenvalues. Moreover, the trace is $3$ and the determinant is $-2$, which are coprime, so $G_3$ is cyclic subgroup separable by Corollary \ref{cor:n=2}.
\end{exmp}

\begin{exmp}
    Let $G_4=\langle a,b,c,t \mid [a,b]=[b,c]=[c,a]=1, tat^{-1}= c^5, tbt^{-1}=a, tct^{-1}=abc \rangle$.
    This corresponds to $A=\langle a,b,c\rangle$, $\varphi(a)=c^5$, $\varphi(b)=a$ and $\varphi(c)=abc$. The integer matrix corresponding to $\varphi$ (with respect to the basis $\{a,b,c\}$ of $A$) is
\begin{equation*}
C_4=
\begin{pmatrix}
0 & 1 & 1 \\
0 & 0 & 1 \\
5 & 0 & 1 
\end{pmatrix}
.
\end{equation*}
The characteristic polynomial is
$$f_\varphi(x)=x^3-x^2-5x-5.$$
It is irreducible over $\Q$ because the substitution $y=x+1$ transforms it to
$$f_\varphi(y)=y^3-4y^2-2,$$
which satisfies Eisenstein's Criterion. Clearly $f_\varphi(x)$ does not reduce to $x^3$ modulo a prime, so we deduce from Theorem \ref{thm:css} that $G_4$ is cyclic subgroup separable.
\end{exmp}

\begin{exmp}\label{lastexample}
Let 
$$G_5= \langle a,b,c,t \mid [a,b]=[b,c]=[c,a]=1, t a t^{-1}= ab^2, tb t^{-1}= a^2 b^2, tct^{-1}=abc \rangle.$$
This corresponds to $A= \langle a,b,c \rangle, \varphi(a)= ab^2, \varphi(b)= a^2 b^2, \varphi(c)= abc$ . The integer matrix corresponding to $\varphi$ (with respect to the basis $\{a,b,c\}$ of $A$) is
\begin{equation*}
C_5=
\begin{pmatrix}
1 & 2 & 1 \\
2 & 2 & 1 \\
0 & 0 & 1
\end{pmatrix}
.
\end{equation*}
The characteristic polynomial factors into irreducible polynomials as follows
$$f_\varphi(x)=(x^2-3x-2)(x-1).$$
These factors satisfy Theorem \ref{thm:css}\ref{item:equivx^ni}, thus $G_5$ is cyclic subgroup separable.
\end{exmp}

\bibliographystyle{alpha}
\bibliography{ref}

\end{document}